%% file: main.tex
\documentclass[11pt]{article}

\usepackage{amssymb,amsthm,amsmath}


\usepackage[usenames,dvipsnames]{xcolor}
\usepackage[colorlinks,citecolor=blue,linkcolor=BrickRed]{hyperref}\usepackage[colorlinks,citecolor=blue,linkcolor=BrickRed]{hyperref}
\usepackage{thmtools,thm-restate}
\usepackage{comment}
\usepackage{cleveref}
\usepackage{fullpage}
\usepackage{pdfpages}
\usepackage{graphicx}
\graphicspath{ {./images/} }

\newtheorem{theorem}{Theorem}[section]

\newtheorem{lemma}[theorem]{Lemma}

\newtheorem{definition}[theorem]{Definition}

\newtheorem{claim}[theorem]{Claim}

\newtheorem{fact}[theorem]{Fact}
\newtheorem*{claim*}{Claim}

\input{macros}

{\hspace*{\fill}$\Box$\par}

\title{Decoupling via Affine Spectral-Independence: \\Beck-Fiala and Koml\'os Bounds Beyond Banaszczyk}
\author{Nikhil Bansal\thanks{University of Michigan, Ann Arbor, MI, USA. \texttt{bansaln@umich.edu }.} 
 \and 
Haotian Jiang\thanks{University of Chicago, Chicago, IL, USA. \texttt{jhtdavid@uchicago.edu}.}}

\date{}
\begin{document}

\allowdisplaybreaks
\begin{titlepage}
\maketitle

\begin{abstract}
\input{abstract}

\end{abstract}

 \thispagestyle{empty}
\end{titlepage}

\thispagestyle{empty}
{\hypersetup{linkcolor=BrickRed}
\tableofcontents
}

\newpage
\setcounter{page}{1}

\input{intro}

\input{overview}

\input{Weak_Beck-Fiala}

\input{Komlos_Proof}

\input{Large_k_Beck_Fiala_Conjecture}

\input{Strong_Beck-Fiala}

\input{conclusions}

\appendix

\input{appendix}

\section*{Acknowledgements}
We thank Thomas Rothvoss and Lap Chi Lau for carefully reading the manuscript and several useful comments.

\bibliographystyle{alpha}
\bibliography{bib.bib}

\end{document}

%% file: macros.tex
\newcommand{\poly}{\ensuremath{\mathsf{poly}}}
\newcommand{\E}{\mathbb{E}}

\newcommand{\R}{\mathbb{R}}

\newcommand{\Tr}{\mathsf{Tr}}

\newcommand{\diag}{\mathsf{diag}}

\global\long\def\norm#1{\left\Vert #1\right\Vert }

\newcommand{\disc}{\mathsf{disc}}
\newcommand{\p}{\mathbb{P}}

\newcommand{\ignore}[1]{{}}

\newcommand{\trunc}{\mathsf{trun}}

%% file: abstract.tex
The Beck-Fiala Conjecture \cite{BF81} asserts that any set system of $n$ elements with degree $k$ has combinatorial discrepancy $O(\sqrt{k})$. A  substantial generalization is the Koml\'os Conjecture, which states that any $m \times n$ matrix with columns of unit $\ell_2$ length has discrepancy  $O(1)$. 

\smallskip

In this work, we resolve the Beck-Fiala Conjecture for $k \geq \log^2 n$. We also give an $\widetilde{O}(\sqrt{k} + \sqrt{\log n})$ bound for $k \leq \log^2 n$, where $\widetilde{O}(\cdot)$ hides $\mathsf{poly}(\log \log n)$ factors. These bounds improve upon the $O(\sqrt{k \log n})$ bound due to Banaszczyk \cite{Ban98}.

\smallskip

For the Koml\'os problem, we give an $\widetilde{O}(\log^{1/4} n)$ bound, improving upon the previous $O(\sqrt{\log n})$ bound \cite{Ban98}. All of our results also admit efficient polynomial-time algorithms.

\smallskip

To obtain these results, we exploit a new technique of ``decoupling via affine spectral-independence'' in designing rounding algorithms.
In particular, our algorithms obtain the desired colorings via a discrete Brownian motion, guided by a semidefinite program (SDP). Besides standard constraints used in prior works, we add some extra {\em affine spectral-independence} constraints, which effectively {\em decouple} the evolution of discrepancies across different rows, and allow us to better control how many rows accumulate large discrepancies at any point during the process. This new technique is quite general and may be of independent interest.

%% file: intro.tex
\section{Introduction}

Combinatorial discrepancy theory studies the following question: given a universe of elements $U=\{1,\ldots, n\}$ and a collection $\mathcal{S} = \{S_1, \ldots, S_m\}$ of subsets of $U$, how well can we partition $U$ into two pieces, so that all sets in $\mathcal{S}$ are split as evenly as possible.
Formally, the combinatorial discrepancy of the set system $\mathcal{S}$ is defined as 
\begin{equation*}
    \disc(\mathcal{S}) := \min_{x: U \rightarrow \{-1,1\}} \max_{i \in [m]} \big|\sum_{j \in S_i}  x(j) \big| ,
\end{equation*}
where the partition $x: U \rightarrow \{-1,1\}$ is also called a {\em coloring}. 
Denoting by $A \in \{0,1\}^{m\times n}$ the incidence matrix of $\mathcal{S}$, i.e., $A_{ij} = 1$ if $j \in S_i$ and $0$ otherwise, we can write $\disc(\mathcal{S}) =  \disc(A) := \min_{x\in\{-1,1\}^n}\norm{{A}x}_{\infty}$. This latter definition also extends to general matrices $A \in \mathbb{R}^{m \times n}$ which may not correspond to any set systems.  Discrepancy is a classical and well-studied topic with many connections and applications to both mathematics and computer science \cite{Cha00,Mat09,CST14}.

\smallskip
\noindent \textbf{The Beck-Fiala and Koml\'os Conjectures.} 
A central and long-standing problem in discrepancy theory has been to understand the discrepancy of bounded-degree set systems --- where each element appears in at most $k$ sets, or equivalently, the matrix $A$ is $k$-column sparse. 
In their seminal work \cite{BF81}, Beck and Fiala showed that $\disc(A) \leq 2k-1$ for any such matrix $A$, and conjectured that $\disc(A) = O(k^{1/2})$. This latter bound is the best possible in general.\footnote{Using probabilistic methods, it is easy to prove an $\Omega(k^{1/2})$ lower bound for the Beck-Fiala problem.}

The Koml\'{o}s problem is a substantial generalization where each column of $A$ has Euclidean length at most $1$, i.e.,~$\sum_{i=1}^m A_{ij}^2 \leq 1$ for all columns $j\in [n]$, and the Koml\'os Conjecture states that $\disc(A)=O(1)$ for any such matrix $A$. 
Notice that the Beck-Fiala problem can be viewed as an instance of the Koml\'os problem where $A_{ij} = k^{-1/2}$ if $j \in S_i$ and $0$ otherwise, and that the Koml\'os Conjecture implies the Beck-Fiala Conjecture.  

\smallskip
\noindent {\bf Previous Results.} The study of these conjectures has led to various powerful and general techniques in discrepancy.
Beck developed the partial coloring method \cite{Bec81}, based on the pigeonhole principle and counting, which was later refined by Spencer \cite{Spe85} and Gluskin \cite{Glu89}. This method shows the existence of a good {\em partial coloring}, where a constant fraction of  elements are colored $\pm 1$, with discrepancy $O(k^{1/2})$ and $O(1)$ for the Beck-Fiala and Koml\'os problems respectively.
Iterating this $O(\log n)$ times gives a full coloring, but loses an additional $\log n$ factor, resulting in $O(k^{1/2} \log n)$ and $O(\log n)$ discrepancy respectively. This method has been studied extensively, but it seems unclear how to use it to get better bounds for these problems.

Banaszczyk developed a different approach  \cite{Ban98} that directly produces a coloring with discrepancy $O(k^{1/2} \log^{1/2} n)$ and $O(\log^{1/2} n)$ for the Beck-Fiala and Koml\'os problems respectively. This remains the best known bound for the Koml\'os problem to date. Banaszczyk's approach uses deep ideas from convex geometry, and gives a beautiful connection between discrepancy and the Gaussian measure of an associated convex body. 
As shown in \cite{DGLN16}, Banaszczyk's bound is equivalent to the existence of a distribution over colorings $x \in \{\pm 1\}^n$  such that the resulting discrepancy vector $Ax$ is $O(k)$- and $O(1)$-subgaussian respectively.\footnote{A mean-zero random vector $y \in \mathbb{R}^m$ is $\sigma^2$-subgaussian, if for all $\theta \in \R^m$, 
it satisfies $\E[\exp(\langle \theta,y\rangle)] \leq \exp(\sigma^2 \|\theta\|_2^2)$.}
The extra $\sqrt{\log n}$ factor loss results from a union bound over all the rows, and seems inherent for techniques that are only based on subgaussianity.

\smallskip
\noindent \textbf{Constructive Methods.} Both these methods were originally non-constructive, and did not give an efficient algorithm to actually find a good coloring. 
However, various efficient algorithms have been developed recently for
both the partial coloring method 
\cite{Ban10,BS13,LM15,HSS14,Rot17,ES18,JSS23} and for Banaszczyk's approach \cite{DGLN16,BDG19,BDGL18,LRR17, BLV22, ALS21, PV23, HSSZ24}, 
leading to various surprising applications in areas such as differential privacy \cite{NTZ13,MN15}, 
combinatorics \cite{BG17,Nik17},
spectral sparsification of graphs \cite{RR20,JRT24,LWZ25}, 
approximation algorithms and rounding \cite{Rot16,BRS22,Ban24}, 
kernel density estimation \cite{PT20,CKW24}, 
randomized controlled trials \cite{HSSZ24}, 
and numerical integration \cite{BJ25a}. 
As this literature is already extensive, 
we refer the reader to the survey \cite{Ban22} for more on these algorithmic aspects.

Very recently, Bansal and Jiang \cite{BJ25b} built on top of these algorithmic developments, and gave a constructive $O(k^{1/2} (\log \log n)^{1/2})$ bound for the Beck-Fiala problem when $k = \Omega(\log^5n)$, almost settling the Beck-Fiala Conjecture (up to a $(\log \log n)^{1/2}$ factor) for such $k$. For smaller values of $k$, however, their approach gave no improvement over previous results, and consequently implied nothing better than Banaszczyk's $O(\log^{1/2} n)$ bound for the Koml\'os problem.

\smallskip
\noindent \textbf{Dimension-Independent Bounds.} All the aforementioned bounds incur a dependence on the dimension $n$. If no dependence on dimension is permitted, the best-known bound for the Beck-Fiala problem is $2k - \log^*k$ due to Bukh \cite{Buk16}, which slightly improves upon the $2k-1$ bound of Beck and Fiala \cite{BF81} and subsequent refinements in \cite{BH97,Hel99}. 
A related result is the celebrated six-standard-deviations result of Spencer \cite{Spe85}, obtained independently by Gluskin \cite{Glu89}, which states that any set system with $m=O(n)$ sets has discrepancy $O(\sqrt{n})$. This can be viewed as a special case of the Beck-Fiala problem, where $k=m=O(n)$.\footnote{This was the only regime where the Beck-Fiala conjecture was previously known to be true.} 
For the general Beck-Fiala problem, however, not even a $(2 - \Omega(1))k$ bound is known to date. Also, no dimension-independent bound is known for the Koml\'os problem, as it would imply the Koml\'os Conjecture.

\smallskip
{\bf Lower Bounds.} It is also possible that the Koml\'os Conjecture may be false. In fact, similar conjectures about $O(1)$ discrepancy for the 3-permutation problem and the Erd\H{o}s discrepancy problem were disproved recently in remarkable results \cite{NNN12,Tao16}. 
For the Koml\'os problem, Hajela conjectured an $\Omega(\log^{1/2} n)$ lower bound \cite{Haj88}, and showed that high-discrepancy instances exist with respect to any subexponential number of sign vectors. This was improved later in \cite{CS21}. 
However, despite significant effort, the best known lower bound for the Koml\'os problem is only $1+\sqrt{2} \approx 2.414$ due to Kunisky \cite{Kun23}. 

There are natural barriers towards proving better lower bounds. Techniques based on SDP duality are ruled out, as the {\em vector discrepancy} for every Koml\'os instance is at most $1$ \cite{Nik13}. Constructions of hard instances based on probabilistic methods also fail, as (semi-)random or even smoothed instances admit small discrepancy \cite{EL19,HR19,FS20,BM20,Pot20,TMR20,BJM+22,ANW22}.

\subsection{Our Contributions}

In this work, we give a general algorithmic approach for obtaining low-discrepancy colorings based on a new technique of {\em decoupling via affine spectral-independence} in designing random walks, 
and use it to obtain several improved bounds for both the Beck-Fiala and Koml\'os problems. 

The framework is quite modular, and for clearer exposition, we present our results in several parts, each highlighting a new idea (that will be explained in \Cref{sec:prelim}).

\medskip
\noindent \textbf{The Basic Method: An Improved Koml\'os Bound Beyond Banaszczyk.}
A basic instantiation of our framework already gives the following improved bound for the Beck-Fiala problem.

\begin{restatable}[Basic Beck-Fiala bound (simple version)]{theorem}{BeckFialaWeak}\label{thm:bf_weak-informal}
Let $A \in \{0,1\}^{m \times n}$ be a set system where each element lies in at most $k$ sets. Then $\disc(A) = \widetilde{O}(k^{1/2} \log^{1/4} n)$, for all values of $k$.
\end{restatable}
This improves Banaszczyk's $O(k^{1/2} \log^{1/2} n)$ bound 
by an $\widetilde{\Omega}(\log^{1/4} n)$ factor for {\em every} $k$. 

Viewing a Koml\'os instance as multiple Beck-Fiala instances at various scales of $k$ (and carefully considering only $O(\log \log n)$ scales, instead of the $O(\log n)$ scales in a naive decomposition), this allows us to obtain  the following improved bound for the Koml\'os problem.

\begin{restatable}[Improved Koml\'{o}s bound]{theorem}{Komlos}\label{thm:Komlos}
Let $A \in \mathbb{R}^{m \times n}$ be a matrix where each column has $\ell_2$ norm at most $1$, then $\disc(A) = \widetilde{O}(\log^{1/4} n)$. 
\end{restatable}
This improves upon the previous best bound of $O(\log^{1/2} n)$  due to Banaszczyk \cite{Ban98}, and refutes Hajela's conjectured lower bound of $\Omega(\log^{1/2} n)$ \cite{Haj88}.

We next describe two extensions of this basic method, and substantially improve the bounds above for the Beck-Fiala problem whenever $k = \widetilde{\Omega}(\log^{1/2} n)$.\footnote{Unfortunately, these extensions do not improve the bound in Theorem \ref{thm:Komlos} for the Koml\'os problem, as they do not improve upon the $O(k)$ Beck-Fiala bound when $k \leq \log^{1/2} n$. We discuss this more in \Cref{sec:conclusion}.}

\medskip
\noindent \textbf{A Multi-Layered Method: Resolving Beck-Fiala Conjecture for $k \geq \log^2 n$.}
Our first extension is a certain {\em multi-layered} version of the method above. This allows us to resolve the Beck-Fiala Conjecture for all $k\geq \log^2 n$ (without any hidden $\log \log n$ factors).

\begin{restatable}[Resolving Beck-Fiala Conjecture for $k \geq \log^2 n$]{theorem}{BeckFialaLargeK}\label{thm:bf_conj_large_k}
Let $A \in \{0,1\}^{m \times n}$ be a set system where each
element lies in at most $k$ sets. If $k = \Omega(\log^2 n)$, then $\disc(A) = O(k^{1/2})$.  
\end{restatable}
Previously, the conjectured $O(k^{1/2})$ bound was only known for $k=\Omega(n)$ \cite{Spe85,Glu89}.\footnote{
While the bound in \cite{BJ25b} comes close, it loses an additional $O((\log\log n)^{1/2})$ factor. It also required that $k \geq \log^5 n$, while Theorem \ref{thm:bf_conj_large_k} holds for $k = \Omega(\log^2 n)$.}

We remark that an $O(k^{1/2})$ bound was not known even if  both the rows and the columns of $A$ are $k$-sparse. 
Here, Lov\'asz local lemma directly gives $\disc(A) = O((k \log k)^{1/2})$, but despite much interest \cite{Sri97,BPRS13,Har15,Kahn}, it was not known how to remove the $O(\log^{1/2} k)$ term.\footnote{In Section \ref{sec:conclusion}, we elaborate more on how this question is related to further progress on the Koml\'os problem.}

\medskip
\noindent \textbf{Adaptive Affine Spectral-Independence: Further Improving the Beck-Fiala Bound.}
Finally, we introduce an {\em adaptive} variant of the basic method which, when combined with the multi-layered method above, further improves upon Theorem \ref{thm:bf_weak-informal} as follows.

\begin{restatable}[Further improvement to Beck-Fiala bound]{theorem}{BeckFialaStrong}\label{thm:bf_strong}
Let $A \in \{0,1\}^{m \times n}$ be a set system where each
element lies in at most $k$ sets. Then $\disc(A) = \widetilde{O}(k^{1/2} + \log^{1/2} n)$.  
\end{restatable}
This  essentially replaces the product in Banaszczyk's $O(k^{1/2} \log^{1/2} n)$ bound \cite{Ban98} by a sum.

Our results also never exceed the $O(k)$ Beck-Fiala bound \cite{BF81}. Thus the bounds above can be summarized as follows as a function of $k$ 
\[\disc(A) =
\begin{cases}
  O(k)   & \text{ if $k \leq  \log^{1/2} n$,}\\
  \widetilde{O}(\log^{1/2} n)  & \text{ if $  \log^{1/2}n \leq k \leq \log n$,}\\
  \widetilde{O}(k^{1/2}) & \text{ if $\log n \leq k \leq \log^2 n$,} \\
  O(k^{1/2}) & \text{ if $k \geq \log^2 n$.} 
\end{cases} 
\]
Notice that this strictly improves upon all previous bounds whenever $k=\widetilde{\Omega}(\log^{1/2} n)$. Specifically, for $k=\log n$, where all the previous approaches only gave $O(\log n)$, the bound here is $\widetilde{O}(\log^{1/2}n)$.

\subsection{Our Approach in a Nutshell} 
We give a detailed overview in Section \ref{sec:prelim}, but the high level approach is the following.
Similar to many prior works on discrepancy, 
our algorithms also evolve a {\em fractional coloring} $x_t \in [-1,1]^n$, starting initially at $x_0 = 0^n$, using cleverly chosen tiny random increments $\Delta x_t$.

Roughly speaking, in the prior works  \cite{BDG19, BG17, BLV22}, the increments $\Delta x_t$  are chosen so that (1) {\em large} rows (of size more than $10 k$ in Beck-Fiala) or those with {\em high} discrepancy are {\em blocked}, so that their discrepancies do not increase, and (2) $\Delta x_t$ is $O(1)$ spectrally-independent, and thus looks random-like to every other row. 

However, these approaches get stuck at $O((k \log n)^{1/2})$ for Beck-Fiala.  The problem is that as the elements get colored over time, the number of rows that can be blocked decreases, 
and essentially, one can only afford to block {\em large} rows.
As the discrepancy evolves randomly for each non-large row, its expected discrepancy is $O(k^{1/2})$, but one ends up incurring an extra $O(\log^{1/2} n)$ factor due to a union bound over the rows. We discuss this more quantitatively in Section \ref{sec:prelim}.

To prevent this union bound loss, our idea is to get more control on the {\em joint} evolution of the discrepancies of various rows.
To do this, we add new SDP constraints which ensure that {\em affine} combinations of $\Delta x_t$, corresponding to certain rows, are also spectrally-independent.
This allows us to {\em decouple} the evolution of the discrepancies of these rows, and achieve a tight control on the number of rows with high discrepancy at {\em every} time, so that we can block them throughout the algorithm. We will discuss this more formally in Section \ref{sec:prelim}.

Recently, Bansal and Jiang \cite{BJ25b} also tried to jointly control the discrepancies of multiple rows, but they used a weaker and more ad hoc approach of blocking certain {\em eigenvectors} of the (slightly modified) input matrix. In contrast, our decoupling via affine spectral-independence framework is much more powerful and modular, and in addition to giving better bounds, it gives simpler algorithms, both conceptually and technically, than those in \cite{BJ25b}.

\smallskip\noindent \textbf{Organization.}
We begin by giving an overview of our approach in \Cref{sec:prelim}. 
As a warm-up, we give a simple algorithmic proof of Banaszczyk's $O(\log^{1/2} n)$ bound for the Koml\'os problem in \Cref{subsec:sdp_bana,subsec:bana_bound_komlos}. Then we motivate and present our new framework for going beyond Banaszczyk's bound by decoupling via affine spectral-independence in \Cref{subsec:beyond_bana}. 
In \Cref{sec:bf_weak}, we describe the basic instantiation of our framework and use it to show \Cref{thm:bf_weak-informal}, followed by the improved Koml\'os bound in \Cref{thm:Komlos} in \Cref{sec:komlos}. Then in \Cref{sec:bf_better}, we present the multi-layered algorithm and prove \Cref{thm:bf_conj_large_k}. Finally, in \Cref{sec:bf_strong}, we describe our adaptive extension and show \Cref{thm:bf_strong}. We conclude in Section \ref{sec:conclusion} with a discussion of the main bottleneck in improving our results further, and state some open questions.

%% file: overview.tex
\section{Preliminaries and Overview of Our Approach}
\label{sec:prelim}

We start with some notation and  the algorithmic framework  that we will use throughout the paper. 

\smallskip
\noindent \textbf{Notation.}
Let $A  \in \R^{m\times n}$ denote the input matrix. We use $a_1,\ldots,a_m$ to denote the rows of $A$, and $a_i(j) = A_{ij}$ to denote its $j$th entry.\footnote{Throughout, the $j$th entry of any vector $u$ will be denoted as $u(j)$.} Throughout, we 
use $i$ to index the rows and $j$ to index columns.
Note that it suffices to upper bound the {\em one-sided} discrepancy $\max_i \langle a_i, x \rangle$ (instead of $\max_i |\langle a_i,x\rangle|$),
which can be done by appending the rows $-a_i$ to $A$ for all $i \in [m]$.\footnote{For the Beck-Fiala problem, this will create $-1$ entries in the matrix $A$, which our algorithms can handle.} 
For a vector $v=(v(1),\ldots,v(n))$, we sometimes use $v^2$ to denote the vector with entries $v(j)^2$ for $j\in [n]$.

\subsection{Algorithmic Framework: SDP-Guided Discrete Brownian Motion}
\label{subsec:basic_framework}
As in many previous algorithmic works on discrepancy \cite{BG17,BDG19,BLV22,BJ25b}, our algorithms start with the coloring $x_0 = 0^n$ and evolve it over time using (tiny) random increments, until some final coloring in $\{-1,1\}^n$ is reached. The time $t$ will range from $0$ to $n$, and starting from $t=0$, it will be updated in discrete increments of size $dt$. It is useful to view $dt$ as infinitesimally small, though we will set $dt = 1/\poly(n)$ so that the algorithm runs in time $\poly(n)$. 

We use the following notation to describe these dynamics.
Let  $x_t \in [-1,1]^n$ denote the {\em fractional} coloring at time $t \in [0,n]$. Let $\mathcal{V}_t := \{j \in [n]: |x_t(j)| \leq 1- 1/(2n)\}$ be the set of {\em alive} elements at time $t$, and $n_t := |\mathcal{V}_t|$ denote the number of alive elements.
At time $t$, the algorithm chooses a random vector $v_t \in \R^{\mathcal{V}_t}$, i.e., only the alive variables are updated, satisfying $\E[v_t]=0$, $\|v_t\|_2=1$ and $x_t \perp v_t$ (and various other properties that will be specified later).
The coloring $x_t$ is updated to $x_{t+dt}= x_t + dx_t$ with
\[d x_t = v_t \sqrt{dt}.\]
Notice that once a variable $j$ is no longer alive (aka dead or frozen), its value $x_t(j)$ stays unchanged, and hence it never becomes alive. Rounding a dead variable to either $1$ or $-1$ incurs $O(1/n)$ error per variable, but this is negligible and will be ignored henceforth. 
Also notice that $\|v_t\|_2=1$ and $x_t \perp v_t$ ensures that $\|x_t\|_2^2=t$ at all steps $t$, and that the process ends by $t=n$. 

\smallskip
\noindent{\bf Choosing $v_t$ via an SDP.}
The power of this framework comes from the ability to choose $v_t$ cleverly and adaptively at each time $t$.
To do this, at each time step $t$ the algorithm computes a PSD matrix $U_t$ by solving a semidefinite program (SDP), and samples $v_t$ with $U_t$ as the covariance matrix (up to scaling, to ensure that $\|v_t\|_2=1$).  
More formally, given the PSD matrix $U_t \in \R^{\mathcal{V}_t \times \mathcal{V}_t}$, let $U_t = Q_t \Lambda_t Q_t^\top$ be its spectral decomposition. We choose
\begin{align} \label{eq:find_vt}
    v_t := (\Tr(U_t))^{-1/2} \, U_t^{1/2} Q_t r_t = (\Tr(U_t))^{-1/2}\, Q_t \Lambda_t^{1/2} r_t,
\end{align}
where $r_t \in \R^{\mathcal{V}_t}$ is a random vector with i.i.d.~Rademacher random variables (taking values $1$ or $-1$ with probability $1/2$ each). Note that $\|v_t\|_2=1$ as  
\begin{align*} 
\|v_t\|_2^2 = v_t^\top v_t = \frac{1}{\Tr(U_t)} r_t^\top \Lambda_t^{1/2} Q_t^\top Q_t \Lambda_t^{1/2} r_t = \frac{\Tr(\Lambda_t)}{\Tr(U_t)} = 1 .
\end{align*}
To avoid notational clutter, we will often write $v_t \in \R^{n_t}$ and $U_t \in \R^{n_t \times n_t}$ instead of $\R^{\mathcal{V}_t}$ and $\R^{\mathcal{V}_t \times \mathcal{V}_t}$, as the relevant $n_t$ coordinates will always be those in $\mathcal{V}_t$.

The SDP used to obtain $U_t$, and hence $v_t$, plays a major role in all previous and our algorithms. 

\subsection{SDP Formulation for Banaszczyk's Bound}
\label{subsec:sdp_bana}

As a warm-up, we first describe the SDP used to algorithmically attain Banaszczyk's bound in a general form below. We then explain the intuition behind the constraints and what they imply.

Let $W \subset \R^h$ be a subspace with dimension $\dim(W) = \delta h$. Consider the following SDP with the matrix variable $U \in \R^{h\times h}$ and parameters $0 < \kappa,\eta \leq 1$.

\begin{align} 
 \text{(SDP for Banaszczyk)} \qquad   
    U_{jj}                                &\leq 1                                         && \text{for all } j \in [h] \label{sdp:jj} \\
    \Tr(U)                                &\geq \kappa h                       \label{sdp:trace}         \\
    \langle ww^\top, U \rangle           &= 0                                          && \text{for all } w \in W  \label{sdp:orthog}\\
    U                                     &\preceq \frac{1}{\eta} \diag(U)              \label{sdp:isotropic}\\                  U &\succeq 0 \label{sdp:psd}
\end{align}
To understand this SDP, consider some feasible solution $U$ and choose a random vector $u = (u(1),\ldots,u(h)) \in \R^h$ with $\E[u]=0$ and covariance $\E[uu^\top]=U$. It is useful to view $u$ as the random coloring update $v_t$ in \Cref{subsec:basic_framework} (up to a scaling factor).

Constraints \eqref{sdp:jj} and \eqref{sdp:trace} are standard and easy to parse. Roughly speaking, they ensure that the random coloring update $u$ is ``uniformly spread out'' on at least $\kappa$ fraction of its coordinates. Constraint \eqref{sdp:trace} also rules out the trivial solution $U=0$.

\smallskip\noindent 
\textbf{Blocking.} The constraints \eqref{sdp:orthog} imply that $\langle u,w\rangle =0$ for all $w\in W$. 
Intuitively, $W$ contains the rows that we would like to {\em block}, and \eqref{sdp:orthog} ensures that the discrepancy of such rows do not increase in the current step. The choice of $W$ for Banaszczyk's bound is specified in \Cref{subsubsec:alg_bana}.

\noindent {\bf Spectral-Independence of Coloring Update.} 
 Let us now consider \eqref{sdp:isotropic}. This is equivalent to saying that $a^T U a \leq (1/\eta) a^T \diag (U) a$ for all vectors $a = (a(1),\ldots,a(h)) \in \R^h$, and hence the random vector $u$ satisfies, 
 \begin{align} \label{eq:sub-isotropic_update}
 \E\left[\langle a, u \rangle ^2 \right] = \E \Big[ (\sum_{j=1}^h a(j) u(j))^2\Big] \leq \frac{1}{\eta} \sum_{j=1}^h a(j)^2 u(j)^2.
 \end{align}
This property of $u$ was referred to as {\em sub-isotropy} previously in \cite{BG17, Ban24}, but we refer to this as {\em spectral-independence} here to be consistent with other literature \cite{AnariLG20}.  
If $\eta=1$, this is the same as saying that the coordinates of the coloring update $u(1),\ldots,u(h)$ are pairwise independent. 
To guarantee the feasibility of the SDP, one can set (say) $\eta = 1/10$, so intuitively \eqref{sdp:isotropic} will ensure that the coloring update $u$ is essentially pairwise independent.

For our purposes, it is useful to think of the parameters $\delta, \kappa, \eta$ as small constants, in which case this SDP is feasible due to the following result. 
\begin{theorem}[\cite{BG17,Ban24}]
\label{thm:sdp_feasibility_bana} 
For any arbitrary subspace $W$ of dimension $\delta h$, the SDP given by the constraints \eqref{sdp:jj}-\eqref{sdp:psd} is feasible
whenever $\delta + \kappa + \eta \leq 1$. 
\end{theorem}

\subsection{Banaszczyk's $O(k^{1/2} \log^{1/2} n)$ Bound for the Beck-Fiala Problem}
\label{subsec:bana_bf_overview}

Let us first see how Banaszczyk's $O(k^{1/2} \log^{1/2} n)$ bound for the Beck-Fiala problem follows using the framework and the SDP above, as we will build on these ideas later. We only describe the main ideas, and defer the technical details to \Cref{subsubsec:alg_bana_komlos} (where we prove Banaszczyk's more general Koml\'os bound in \Cref{lem:algo_bana_komlos}). 
\subsubsection{An Algorithm for Banaszczyk's Beck-Fiala Bound}
\label{subsubsec:alg_bana}
Let $A \in \{0,1\}^{m \times n}$ be the input matrix with $k$-sparse columns. Consider the following algorithm using the framework and SDP from \Cref{subsec:basic_framework,subsec:sdp_bana}. 

At each time $t$, repeat the following until $n_t \leq  10$:
\begin{enumerate}
\item Call a row $i$ {\em large} (at time $t$)  if its size, i.e., the number of alive elements $|\{j \in \mathcal{V}_t: a_i(j) = 1\}|$, is more  than $10 k$. 
Let $W_t$ be the span of all the large rows and the single vector $x_t$. 
    \item Solve the SDP in \eqref{sdp:jj}-\eqref{sdp:psd} with $W = W_t$, $\kappa = 1/4$ and $\eta=1/4$. Use the resulting solution $U_t$ to find $v_t$ as in \eqref{eq:find_vt}, and update the coloring $x_t$ by $v_t \sqrt{dt}$.
\end{enumerate}
Note that at any time $t$, there can be at most $n_t/10$ large rows (as the column sparsity is $k$), and thus \Cref{thm:sdp_feasibility_bana} implies that the SDP is always feasible.

\subsubsection{Main Ideas Behind the Analysis}
\label{subsubsec:analysis_Banaszczyk}
The analysis relies on two crucial observations:
\begin{enumerate}
    \item As the algorithm always blocks all the large rows with sizes more than $10k$, one can pretend that each row has size $\leq 10k$ (as any row incurs zero discrepancy while it is large). 
    \item As the coloring update $v_t$ is spectrally-independent,   the almost pairwise independence 
    condition in \eqref{eq:sub-isotropic_update} gives that the discrepancy of each row $i$ has essentially {\em subgaussian} tails. 
\end{enumerate}
More quantitatively, these observations give that $\disc_t(i) := \langle a_i, x_t \rangle$ is essentially $O(k)$-subgaussian (see \Cref{lem:algo_bana_komlos} for details) with tail bound 
\begin{align} \label{eq:disc_tail_bound_overview}
    \p(\disc_t(i) \geq c k^{1/2}) \approx \exp(- c^2) .
\end{align}
Now setting $c := O(\log^{1/2} n)$ and taking a union bound over all rows, recovers Banaszczyk's $O(k^{1/2} \log^{1/2} n)$ bound for the Beck-Fiala problem.

\subsubsection{Where Can We Improve Upon Banaszczyk's Bound?}
\label{subsubsec:improve_Bana}
At first sight, the algorithm above seems wasteful as it only blocks large rows (even though, in principle, it could have also blocked, say, $n_t/10$ high discrepancy rows at any time $t$).

However, it is unclear to exploit this to improve upon the $O(\log^{1/2} n)$ loss above. For example, suppose we set $c= \log^{0.49} n$. Then, 
by the tail bound in \eqref{eq:disc_tail_bound_overview}, there can be up to $m \exp(-c^2) \gg n^{0.999}$ rows with discrepancy $\Omega(ck^{1/2})$ (already by time, say, $n/2$). 
But now as the algorithm progresses and $n_t$ drops below $O(m\exp(-c^2))$, one can no longer afford to block even a reasonable fraction of these high discrepancy rows.
 
This was a key bottleneck in almost all previous approaches. In particular, as we can only guarantee that  each row individually evolves randomly, and as the discrepancies of different rows can be highly correlated, it is unclear how to beat the union bound loss in the analysis above.

\smallskip
\noindent \textbf{Independent Evolution of Rows Helps.} 
Interestingly, it turns out that if the discrepancies of the different rows evolve {\em independently}, then one can do substantially better. 

To demonstrate this concretely, let us assume that $k = \log n$ for the rest of this overview. We will show  a  $\mu := O(\log^{2/3} n)$ bound (assuming independent evolution of rows), improving upon Banaszczyk's $O(\log n)$ bound.

Call a row {\em dangerous} if its discrepancy is more than $\mu$. Fix some column $j$, and consider the $k$ non-zero entries in it. As the rows evolve independently, using \eqref{eq:disc_tail_bound_overview} with $c = \Theta(\log^{1/6} n)$ (so that $\mu = c k^{1/2}$), the probability that more than $\mu$ rows in this column become dangerous is at most 
\begin{align} \label{eq:tail_bound_ind}
\approx \binom{k}{\mu} \cdot \big( \exp(-c^2) \big)^{\mu} \approx \exp(- \mu c^2) = \frac{1}{\poly(n)} .
\end{align}
By a union bound over all columns and time steps, we can assume that (whp) every column has at most $\mu$ dangerous rows at any time. 
This implies that no matter {\em which} $n_t$ columns are alive at time $t$,
there can be at most $\mu n_t$ entries in dangerous rows, among those alive columns. So there can be at most $n_t/10$ dangerous rows with sizes more than $10 \mu$, which we can easily block.

Note that there can still be many {\em tiny} dangerous rows with sizes less than $10\mu$, but we can simply ignore them all as they cannot incur more than $10\mu$ additional discrepancy henceforth!
Together, this gives the promised $O(\mu) = O(\log^{2/3} n)$ discrepancy bound. 

To summarize this important point, in an ideal scenario where the discrepancies of rows evolve {\em independently}, we can ensure that
\begin{align*}
(*) \quad &\text{At any time $t$, even though $\gg n_t$ rows may have high discrepancy},\\ & \text{all but $\ll n_t$ of them will be negligibly tiny}.
\end{align*}

\smallskip
\noindent 
{\bf Handling General Instances.}
Of course, this independent evolution of rows is too good to hope for in general (e.g., if two rows are identical, their discrepancies will also be identical).
Nonetheless, we will show below how to {\em decouple} the row discrepancies, so that the ideal scenario above holds approximately. For concreteness, we show next an $O(\log^{3/4} n)$ bound (for $k=\log n$) using some basic ideas. Finally, we sketch the other refinements used to improve the bounds even further.

\subsection{Beyond Banaszczyk: Decoupling via Affine Spectral-Independence} 
\label{subsec:beyond_bana}
We now present our  approach to approximately decouple the discrepancies of various rows, and show how it gives 
an $O(\log^{3/4} n)$ bound for general Beck-Fiala instances with $k=\log n$.

\subsubsection{Affine Spectral-Independence}
\label{subsubsec:affine_spectral-independence}

Again, we start by describing a general formulation of the new SDP that we use.

Let $W \subset \R^h$ be a subspace with dimension $\dim(W) = \delta h$, and let $E_s \in \mathbb{R}^{r_s h \times h}$ be  a collection of matrices for $s=1,2,\ldots,q$, where $r_s \geq 1$ for each $s\in [q]$. 
Consider the following SDP with the matrix variable $U \in \R^{h\times h}$ and parameters $0\leq \kappa,\eta,\eta_s \leq 1$.
\begin{align} 
 \text{(Our SDP)} \qquad   
    U_{jj}                                &\leq 1                                         && \text{for all } j \in [h] \label{sdp:jj_new} \\
    \Tr(U)                                &\geq \kappa h                       \label{sdp:trace_new}         \\
    \langle ww^\top, U \rangle           &= 0                                          && \text{for all } w \in W  \label{sdp:orthog_new}\\
    U                                     &\preceq \frac{1}{\eta} \diag(U)              \label{sdp:isotropic_new}\\
    E_sUE_s^\top                          &\preceq \frac{r_s}{\eta_s} \diag(E_s U E_s^\top) && \text{for all } s \in [q] \label{sdp:pairwise-disc}\\
    U                                     &\succeq 0 \label{sdp:psd_new}
\end{align}

Compared to the SDP in \eqref{sdp:jj}-\eqref{sdp:psd}, the key new constraints are \eqref{sdp:pairwise-disc} which we explain below.\footnote{Strictly speaking, \eqref{sdp:isotropic_new} is a special case of \eqref{sdp:pairwise-disc} (with $E_s=I$), but it is useful to consider them separately as they will imply different properties and are used separately in the analysis.}  Fix some matrix $E_s$, and consider the random vector $d = E_s u \in \R^{r_sh}$, where as before, $u$ is a mean-zero random vector (corresponding to the coloring update $v_t$) whose covariance matrix is $U$. 
Its $i$-th entry $d(i) = \langle E_{s,i},u \rangle $,  where  $E_{s,i}$ is the $i$-th row of $E_s$.

Constraint \eqref{sdp:pairwise-disc} is equivalent to saying that for every vector $\gamma\in \R^{r_sh}$, we have
\begin{align} \label{eq:affine_subisotropy_overview}
\E \left[ (\gamma^T E_s u)^2 \right] = \E \Big[ (\sum_{i=1}^{r_s h} \gamma(i) d(i))^2 \Big]     \leq \frac{r_s}{\eta_s} \sum_{i=1}^{r_sh} \gamma(i)^2 \,\, \E[ d(i)^2]  =  \frac{r_s}{\eta_s} \sum_i \gamma(i)^2  \E[\langle E_{s,i},u \rangle^2]. 
\end{align}

\smallskip
\noindent {\bf Affine Spectral-Independence.}
We call this property of $u$ as being $(r_s/\eta_s)$-{\em affine spectrally-independent} with respect to the matrix $E_s$. 
Equivalently, the random vector $E_s u$ is $(r_s/\eta_s)$-spectrally-independent.

Intuitively, if the rows of $E_s$ are directions whose discrepancies we wish to control, then the vector $E_s u$ is the discrepancy update when the coloring is updated by $u$. So \eqref{sdp:pairwise-disc} says that the discrepancies of rows of $E_s$ evolve $O(r_s/\eta_s)$-pairwise independently.

We will choose the matrices $E_s$ above  to correspond to certain (slightly modified) rows of the input matrix $A$.  In general however, the $E_s$ can be chosen arbitrarily. This makes the framework quite powerful and potentially useful in other scenarios as well.

As an extension of \Cref{thm:sdp_feasibility_bana}, we have the following feasibility result for the new SDP. The proof of this result is deferred to Appendix \ref{subsec:sdp_feasibility}.

\begin{restatable}
{theorem}{SubIsoVecGeneral}\label{thm:sdp_feasibility} 
For any arbitrary subspace $W$, and any matrices $E_1,\ldots,E_q$, the SDP given by the constraints \eqref{sdp:orthog_new}-\eqref{sdp:psd_new} is feasible
whenever $\delta + \kappa + \eta +  \sum_{s=1}^q \eta_s \leq 1$. 
\end{restatable}

One may wonder if the extra factor of $O(r_s)$ in \eqref{sdp:pairwise-disc} is really needed. In \Cref{subsubsec:necessity_r_s} we give a simple example to show that this is indeed the case.

\subsubsection{Decoupling Bound for Discrepancy and the Basic Method} 
\label{subsub:decoupling}

We now show a basic version of this approach that gives $\mu := O(\log^{3/4} n)$ discrepancy for $k = \log n$. 

Here, we will choose only one matrix $E_s$ (so $q = 1$). Consider some time $t$. As before, {\em large} rows with sizes $>10k$ are always blocked, and {\em tiny} rows with sizes $<10\mu$ can be ignored.
So let $\mathcal{M}_t$ denote 
the set of all {\em medium} (or {\em interesting}) rows with sizes in the range $[10 \mu, 10k]$. 
As each medium row has size at least $10\mu$, we have $|\mathcal{M}_t| \leq k n_t/(10\mu)$. 
We set $E_1$ to be the matrix consisting of all medium rows $\mathcal{M}_t$, and as $r_1 =|\mathcal{M}_t|/n_t \leq k/10\mu$, the parameter $r_1/\eta_1= \Theta(k/\mu)$ in \eqref{sdp:pairwise-disc}.

Call a medium row {\em dangerous} if its discrepancy is more than (our target) $\mu$. 
We view $\mu$ as $\mu = c k^{1/2}$
with $c = \Theta(\log^{1/4} n)$.
We will show that at {\em every} time $t$, at most $n_t/10$ rows can be dangerous, and hence they can all be blocked.

Fix some column $j$, and consider the $k$ non-zero entries in it. As the discrepancies of medium rows evolve using $\Theta(k/\mu)$-affine spectrally independent updates (as in \eqref{eq:affine_subisotropy_overview}), we can show that the probability that more than $\mu$ rows in this column become dangerous is (roughly)
\begin{align*}
        \big( \exp(- c^2) \big)^{\mu \cdot (\eta_1/r_1)} =  \exp( - \Theta(c^2\mu^2/k)) = 1/\poly(n) .
 \end{align*}
where the last step uses that $c^2 \mu^2 = c^4 k = \Theta(k \log n)$, by our choice of parameters.

The formal statement of the decoupling bound above is more technical and is given in Theorem \ref{thm:general_freedman}, and proved in \Cref{subsec:proof_general_freedman}.
Intuitively, it is analogous to the tail bound in \eqref{eq:tail_bound_ind} for the independent case, but loses the extra affine spectral-independence parameter $(r_1/\eta_1)$ in the exponent.

Putting this all together, the dangerous rows can be blocked at all times, and arguing as in Section \ref{subsubsec:improve_Bana} gives that the final discrepancy is $O(\mu)$ with high probability.

\subsubsection{Extensions of the Basic Method}
 
We now briefly mention a few other ideas required for the improved Koml\'os bound in \Cref{thm:Komlos}, and our stronger Beck-Fiala bounds in \Cref{thm:bf_conj_large_k,thm:bf_strong}. 

\smallskip
\noindent \textbf{Our Improved Koml\'os Bound.} Roughly speaking, we decompose the matrix $A$ into multiple Beck-Fiala instances at various {\em scales} based on the magnitudes of its entries, and apply \Cref{thm:bf_weak-informal} to these instances in parallel, incurring an $\widetilde{O}( \log^{1/4} n)$ discrepancy for each of them. Doing this naively would lose an extra $O(\log n)$ factor in the overall discrepancy due to the $O(\log n)$ scales, but a key part of the argument is to show that $O(\log \log n)$ scales suffice. See \Cref{sec:komlos} for details.

\smallskip
\noindent \textbf{A Multi-Layered Method.} In the basic method above, we need to set the affine spectral-independence parameter $r_1/\eta_1$ rather conservatively, as we simply block all the dangerous rows whose sizes are more than $10\mu$. However, instead of just naively blocking them, we can try to run a smarter algorithm on top whose job is to control the discrepancies of these dangerous rows.

Building on this idea iteratively leads to a natural {\em multi-layered} method: the rows that become dangerous at level $1$ are fed into an algorithm running at level $2$, and rows which become dangerous again there are fed into an algorithm running at level $3$ and so on, where the sparsity parameter $k$ improves as the levels get higher (and hence the discrepancy loss does not accumulate much over the levels). 
The details of this multi-layered method are in \Cref{sec:bf_better}.

\smallskip
\noindent \textbf{Adaptive Affine Spectral-Independence.} In the multi-layered method above, the affine spectral-independence parameters $r_s/\eta_s$ for the different algorithms (at different levels)
were set to fixed values. One can extend this method even further, by allowing the algorithm to vary these parameters for each row individually over time in a dynamic manner, depending on the number of currently alive elements in that row. 
We defer the details of this adaptive method to \Cref{sec:bf_strong}.

%% file: Weak_Beck-Fiala.tex
\section{A Basic Algorithm for the Beck-Fiala Problem}
\label{sec:bf_weak}
We start with the most basic application of our framework and use it to 
prove \Cref{thm:bf_weak-informal} 
for the Beck-Fiala problem.
For technical reasons, we will work with the following (stronger) statement. 

\begin{restatable}[Basic Beck-Fiala bound]{theorem}{BeckFialaWeakb}
\label{thm:bf_weak}
Let $A \in \{0, \pm 1\}^{m \times n}$ be a matrix where each column has at most $k$ non-zeros. Then 
there is a polynomial time algorithm that finds a coloring with
$\disc(A) = O(b)$, where $b$ is the smallest number that satisfies 
\[
b^2 \max(1, b/\log n) \gg k (\log n \log \log n)^{1/2} \quad \text{and} \quad b \gg (k \log \log n)^{1/2} .
\]
\end{restatable}
Our algorithm will also always satisfy the $O(k)$ Beck-Fiala bound \cite{BF81}. Considering the various regimes of $k$ in Theorem \ref{thm:bf_weak}, a direct computation gives the following explicit bounds.
\[\disc(A) =
\begin{cases}
  O(k)   & \text{ if $k \leq  \log^{1/2} n$}\\
   \widetilde{O}( k^{1/2}  \log^{1/4} n)  & \text{ if $ \log^{1/2} n   \leq k  \leq \log^{3/2}n$}\\
  \widetilde{O}(k^{1/3} \log^{1/2} n)  & \text{ if $  \log^{3/2}n \leq k \leq \log^3 n$}\\
  O(k^{1/2})  (\log \log n)^{1/2} & \text{ if $k \geq \log^{3}n$}. 
\end{cases}
\]
Notice that for $k \geq \log^3 n$, this bound already comes within a $(\log \log n)^{1/2}$ factor of the conjectured $O(k^{1/2})$ bound.
Moreover, these bounds are never worse than $\widetilde{O}(k^{1/2} \log^{1/4}n)$ (for every $k$), and hence \Cref{thm:bf_weak}  implies \Cref{thm:bf_weak-informal}. 
 
We now describe the algorithm.

\subsection{The Basic Algorithm}
\label{subsec:alg_bf_weak}
Let $A \in \{0, \pm 1\}^{m \times n}$ be an input instance of the Beck-Fiala problem,  where each column has at most $k$ non-zeros.
We will show that $\disc(A)=O(b)$, where $b$ 
satisfies the conditions in \Cref{thm:bf_weak}. For now, we ignore these conditions and just view $b$ as some target bound in the range $[\sqrt{k}, k]$.

We first give some definitions and note some simple observations.

\smallskip
\noindent {\bf Large, Tiny, and Medium Rows.}
At any time $t$, the {\em size} of a row $i$ is the number of non-zero entries of $a_i$ in the alive elements, i.e.,~$|\{j \in \mathcal{V}_t: a_i(j) \in \{ \pm 1\}\}|$. Observe that the size of a row can only decrease over time.

We say that row $i$ is {\em large} at time $t$, if it has size more than $10k$.

We say that row $i$  is {\em tiny} or {\em irrelevant}, if it has size less than \[\mu := \max(b, b^2/\log n).\] 
This parameter $\mu$ will play an important role.

We call a row  {\em medium} or {\em interesting} if has size in the range $[\mu, 10k]$.

A simple observation that we will use repeatedly is that, both large and tiny rows can essentially be ignored while bounding their contributions to discrepancy. Indeed, there can be at most $n_tk/10k = n_t/10$ large rows at any time $t$ (as $A$ restricted to columns in $\mathcal{V}_t$ has at most $k n_t$ non-zero entries), and
so the algorithm can always choose $v_t$ to be orthogonal to each large row (i.e., add them to the subspace $W$ in the SDP), and hence ensure that
$\langle a_i,x_t\rangle =0 $ as long as row $i$ stays large.

Similarly, we can simply ``forget about'' a row once it becomes tiny, as by Banaszczyk's bound in \Cref{lem:algo_bana_komlos}, 
any such row can incur at most  $O((\mu \log n)^{1/2})$ (or, trivially, at most $O(\mu)$) discrepancy henceforth. Our choice of $\mu$ exactly satisfies $\min(\mu, (\mu \log n)^{1/2}) = b$, and thus with high probability, every tiny row incurs additional discrepancy at most $O(b)$.

Thus, medium rows are the only ones whose discrepancies need to be carefully controlled.
We denote the set of such rows at time $t$ by $\mathcal{M}_t$. As each column $j$ has sparsity $k$ and each medium row has size at least $\mu$, we have the bound 
\begin{equation}
\label{eq:mt-bound}
|\mathcal{M}_t| \leq n_t k/\mu.
\end{equation}
In particular, 
    $|\mathcal{M}_t| \leq n_t \log n$
(as $\mu \geq k/\log n$, since $\mu \geq b^2/\log n$ by definition and our target $b \geq \sqrt{k}$).

\smallskip
\noindent {\bf The Matrix $E_t$.} Let $\beta := b/(10k) \leq 0.1$. 
At time $t$,
define $E_t \in \R^{|\mathcal{M}_t| \times n_t}$ to be the matrix with rows $E_{t,i}$ corresponding to all medium rows $i \in [|\mathcal{M}_t|]$, with entries 
\[E_{t,i}(j) = a_i(j) - 2 \beta  a_i(j)^2 x_t(j).\]
The matrix $E_t$ will be useful for controlling the discrepancies of the rows, which will become clear later in \Cref{subsubsec:bound_disc_bf_weak}. 

\noindent {\bf Dangerous Rows.} We call a row $i$ {\em dangerous} at time $t$, if (i) it is medium and (ii) its current discrepancy $\langle a_i, x_t \rangle \geq 2b$. 

\smallskip 

We now describe the algorithm formally.

\smallskip
\noindent \textbf{The Algorithm.} Let $A \in \{0,\pm 1\}^{m \times n}$ be an input matrix with column-sparsity $\leq k$. 
Consider the following algorithm using the framework from \Cref{sec:prelim}, and the notation above.

At each time $t$, repeat the following until $n_t \leq  10$.
\begin{enumerate}
\item  Let $W_t$ denote the subspace spanned by 
(i) all the large rows $a_i$, (ii) rows $E_{t,i}$ for all dangerous rows $i$, and (iii) the (single) vector $x_t$. If $\dim(W_t) > n_t/3$, declare FAIL. 
    \item Solve the SDP \eqref{sdp:jj_new}-\eqref{sdp:psd_new} with the subspace $W = W_t$, the matrix $E_1=E_t$ (here $q=1$, and we have only one such matrix $E_s$) and parameters  $\delta=1/3, \kappa = 1/6$ and $\eta= \eta_1 = 1/4$.
    Use the resulting SDP solution $U_t$ to find $v_t$ as in \eqref{eq:find_vt}, and update the coloring $x_t$ by $v_t \sqrt{dt}$.
\end{enumerate}

\subsection{The Analysis}
We now analyze the algorithm and prove \Cref{thm:bf_weak}.

Notice that if the algorithm does not declare FAIL in Step 1 at time $t$, then $\dim(W_t)\leq n_t/3$ (i.e.,~$\delta \leq 1/3$), and hence  the SDP in Step 2 always has a feasible solution by \Cref{thm:sdp_feasibility}, as $\delta + \kappa + 2 \eta \leq 1$ by the choice of the parameters.

\smallskip 
\indent {\bf Overview.} 
To prove \Cref{thm:bf_weak}, we will show that with high probability, 
(i) the algorithm never declares FAIL during its execution, and (ii) each row has discrepancy $O(b)$.

To prove (i), we will show that with high probability, at any time $t$, the number of dangerous rows never exceeds $n_t/10$. As the number of large rows at any time is at most $n_t/10$, this will imply that $\dim(W_t) \leq n_t/10 + n_t/10 + 1 \leq n_t/3$ as desired.
This is the part that requires new ideas and is harder, and is described in Section \ref{subsubsec:num_dang_rows_bf_weak} - \ref{subsubsec:dang_row_col_proof_bf_weak}. 

To prove (ii), it suffices to bound the discrepancy incurred by a row while it is medium.
As discussed before, this is because (1) a row incurs no discrepancy as long as it is large, as $v_t$ is always orthogonal to every large row,  and 
(2) as our SDP solution satisfies $U_t \leq 4 \diag(U_t)$ at all $t$, 
by \Cref{lem:algo_bana_komlos}, we always have the Banaszczyk guarantee. In particular,
once any row becomes tiny (and hence of size at most $\mu$), it can incur at most $O(\min(\mu,(\mu \log n)^{1/2})) = O(b)$ discrepancy henceforth, with high probability.
 
\subsubsection{Bounding the Discrepancy of Medium Rows} 
\label{subsubsec:bound_disc_bf_weak}

We now bound the discrepancy of medium rows. As discussed previously, it suffices to only bound the one-sided discrepancy $\max_i \langle a_i, x \rangle$.
To do so, instead of tracking the actual discrepancy of medium rows $i$, we will work with the {\em regularized} discrepancy
\begin{equation}
\label{eq:reg_disc}
Y_t(i)  := 
    \langle a_i, x_t \rangle  + \beta G_t(i)  ,
\end{equation}
where $\beta = b/(10k)$ and $G_t(i) = \sum_{j=1}^n a_i(j)^2 (1 - x_t(j)^2)$ is the {\em energy} of row $i$ at time $t$.

Intuitively, the reason for considering the regularized discrepancy $Y_t(i)$ instead of the actual discrepancy $\langle a_i, x_t \rangle$ is because $Y_t(i)$ admits a {\em negative drift}, as is shown below in \eqref{eq:dZ_t}, whereas $\langle a_i, x_t \rangle$ is only a martingale since $\E[d x_t] = 0$. As we will see later in \Cref{subsec:decoupling_bounds,subsec:proof_general_freedman}, the negative drift of $Y_t(i)$ will be crucial for us to decouple their evolutions.

As $G_t(i)\geq 0$, $Y_t(i)$ upper bounds the discrepancy $\langle a_i,x_t\rangle $.
Moreover, as $G_t(i)\leq 10 k$ for a non-large row, 
the choice of $\beta$ ensures that $Y_t(i) \leq \langle a_i,x_t\rangle + b$.
In particular, $Y_{t_i^*}(i) \leq b$ 
when a row $i$ first becomes medium at time $t_i^*$ (we set $t_i^* = 0$ if row $i$ is never large, and $t_i^* = n$ if it starts tiny).

Let $Y_t$ denote the vector of regularized discrepancies.
We think of the vector $Y_t$ as defined only on the coordinates in $\mathcal{M}_t$, i.e., $Y_t \in \mathbb{R}^{\mathcal{M}_t}$.  By design, the algorithm has the following useful property.

\begin{lemma} \label{lem:reg_disc_bounded_bf_weak}
    If the algorithm does not FAIL, then at each time $t$, each medium row $i$ has $Y_t(i) \leq 3b$.
\end{lemma}
\begin{proof}
    As $d x_t = v_t \sqrt{d t}$,  the increment of $Y_t(i)$ is given by
\begin{align}
 d Y_t(i) 
    & = \langle a_i, v_t \sqrt{dt}\rangle  - \beta \sum_{j=1}^n a_i(j)^2\left( \big( x_t(j) + v_t(j) \sqrt{dt}\big)^2 - x_t(j)^2\right) \nonumber \\
    & = \langle E_{t,i},v_t\rangle \sqrt{dt} - \beta \langle a_i^2 , v_t^2 \rangle dt.\label{eq:dZ_t}
\end{align}
where we use $a_i^2$ and $v_t^2$ to denote the vectors with entries $a_i(j)^2$  and $v_t(j)^2$, and as the entries of $E_{t,i}$ are exactly $a_i(j)- 2 \beta a_i(j)^2 x_t(j)$.

As the subspace $W_t$ contains the vector $E_{t,i}$ for all dangerous rows $i$ (assuming the algorithm does not FAIL), $Y_t(i)$ can only decrease deterministically for dangerous rows. As a medium row becomes dangerous when its discrepancy exceeds $2b$, its regularized discrepancy can never exceed $3b$ (plus possibly a $O(\sqrt{dt})$ term due to change in discrepancy from time $t-dt$ to $t$, which is negligible).
\end{proof}

To summarize, we have shown that if the algorithm does not FAIL, then with high probability, each row has discrepancy $O(b)$. 

\subsubsection{Bounding the Number of Dangerous Rows}
\label{subsubsec:num_dang_rows_bf_weak}

We now show that with high probability, the algorithm does not FAIL.
As the algorithm can only FAIL at time $t$ if the number of dangerous rows exceeds $n_t/10$, it suffices to show the following.
\begin{claim} \label{cl:no_bad_bf_weak}
If $b$ satisfies that (i) $b \mu \gg k (\log n \log \log n)^{1/2}$ and (ii) $b \gg (k \log \log n)^{1/2}$.
Then with high probability, at all times $t$, there are at most $n_t/10$ dangerous rows.
\end{claim}
For each column $j$, let $\mathcal{C}_j$ denote the set of (at most $k$) rows $i$ with $a_i(j)\neq 0$.
To prove Claim \ref{cl:no_bad_bf_weak},
it suffices to show the following guarantee for each column.

\begin{restatable}[Dangerous rows per column]{lemma}{DangRowsPerCol}
\label{lem:per-column} If $b$ satisfies that (i) $b \mu \gg k (\log n \log \log n)^{1/2}$ and (ii) $b \gg (k \log \log n)^{1/2}$, then 
    with high probability, for all times $t$ and all  columns $j\in [n]$,
at most $\mu/10$ rows in $\mathcal{C}_j$ can become dangerous.
\end{restatable}
Lemma \ref{lem:per-column} implies Claim \ref{cl:no_bad_bf_weak} for the following reason. As each dangerous row has size at least $\mu$, no matter which subset of the columns $\mathcal{V}_t$ are currently alive at time $t$, Lemma \ref{lem:per-column} implies that at most $ (\mu/10) n_t/\mu  = n_t/10$ rows can be dangerous.

\subsubsection{A Decoupling Bound via Affine Spectral-Independence}
\label{subsec:decoupling_bounds}

To prove Lemma \ref{lem:per-column}, for each column $j$, we will need to simultaneously track the entries $Y_t(i)$ for all $i\in \mathcal{C}_j$. Here, the new affine spectral-independence SDP constraints \eqref{sdp:pairwise-disc} will be crucial. We will develop a general decoupling bound for such random processes. We describe this next.

We begin with a useful definition.
\begin{definition}[$(\alpha,\theta)$-pairwise independence]
\label{defn:weak_pairwise_ind}
Let $\alpha \geq  1$ and $\theta \geq 0$ be parameters. A random vector $X \in \mathbb{R}^m$ is 
$(\alpha,\theta)$-pairwise independent if it satisfies:
\begin{enumerate}
    \item[(i)] (Almost Pairwise Independence) For any $\gamma \in \mathbb{R}^m$, one has $\E[ \langle \gamma, X \rangle^2 ] \leq \alpha \cdot \E [\langle \gamma^2, X^2 \rangle ]$, and
    \item[(ii)] (Coordinate-Wise Negative Drift) $\E[X(i)] \leq  - \theta \cdot \E[X(i)^2]$ for each $i\in [m]$. 
\end{enumerate}
\end{definition}

\smallskip
\noindent \textbf{Affine Spectral-Independence Implies $(\alpha, \theta)$-Pairwise Independence of $d Y_t$.} We show that the increment $d Y_t :=  Y_{t+ dt} - Y_t$ of the regularized discrepancy $Y_t$ defined in \eqref{eq:reg_disc} (for medium rows) is $(\alpha, \theta)$-pairwise independent with $\alpha =4k/\mu$ and $\theta = \beta/5$. 

\begin{lemma} 
\label{lem:weak_pairwise_bf_weak}
The increment $d Y_t \in \mathbb{R}^{\mathcal{M}_t}$ at any time $t$ is $(4k/\mu, \beta/5)$-pairwise independent. 
\end{lemma}
\begin{proof} We first show the coordinate-wise negative drift property of $d Y_t$. Fix some row $i$. Then, as $dY_t(i)  = \langle E_{t,i},v_t\rangle \sqrt{dt} - \beta \langle a_i^2 , v_t^2 \rangle dt $ by \eqref{eq:dZ_t}, and  $\E[v_t]=0$, we have
\[
\E[d Y_t(i)] = - \beta \E \langle a_i^2 , v_t^2 \rangle dt.\]
Similarly, squaring $d Y_t(i)$ and ignoring  lower-order $O((dt)^{3/2})$ terms, and taking expectation gives
\begin{align*}
\E[(d Y_t(i))^2] 
& = \E \langle E_{t,i}, v_t \rangle^2 d t \leq 4 \cdot \E \langle E_{t,i}^2, v_t^2 \rangle dt  \leq 5 \cdot \E \langle a_i^2 , v_t^2 \rangle dt, 
\end{align*}
where the first inequality uses the SDP constraints \eqref{sdp:isotropic} with our choice of $\eta = 4$, and the last inequality uses that as $\beta = b/10k \leq 0.1$ and $|a_i(j)|\leq 1$, the entries $E_{t,i}(j) = a_i(j) - 2\beta a_i(j)^2 x_t(j)$ satisfy $|E_{t,i}(j)| \leq 1.2 |a_i(j)|$.  

Thus $dY_t$ satisfies condition (ii) in \Cref{defn:weak_pairwise_ind} with $\theta = \beta/5$.

\smallskip
We now show the almost pairwise independence property of $d Y_t$. Fix some vector $\gamma \in \R^{|\mathcal{M}_t|}$. Using the expression for $d Y_t$ and ignoring the lower-order $O((dt)^{3/2})$ terms, we have
\[
\E \langle \gamma, d Y_t \rangle^2  = \gamma^\top \E \Big[E_t v_t v_t^\top E_t^\top \Big] \gamma  dt \quad \text{ and } \quad  \E \langle \gamma^2, (d Y_t)^2 \rangle  =  \sum_{i \in \mathcal{M}_t} \gamma_i^2 \cdot  \E \langle E_{t,i}, v_t \rangle^2 dt
\]
Now, by the affine spectral-independence constraints \eqref{sdp:pairwise-disc}, we have that
\begin{align} \label{eq:almost_pairwise_bf_weak}
\gamma^\top \E \Big[E_t v_t v_t^\top E_t^\top \Big] \gamma \leq \frac{4 k}{\mu} \sum_{i \in \mathcal{M}_t} \gamma_i^2 \cdot  \E \langle E_{t,i}, v_t \rangle^2 ,
\end{align}
as we choose the parameter $\eta_1 = 1/4$ in the SDP in the algorithm, and as $E_1 = E_t$ is an $|\mathcal{M}_t|\times n_t$ matrix, so the quantity  $r_1 = |\mathcal{M}_t|/n_t \leq k/\mu$ by the upper bound on $|\mathcal{M}_t|$ in \eqref{eq:mt-bound}.

This gives that $d Y_t$ satisfies condition (i) in \Cref{defn:weak_pairwise_ind} with $\alpha = 4k/\mu$. 
\end{proof}

\smallskip
\noindent{\bf A Decoupling Bound for $(\alpha,\theta)$-Pairwise Independent Processes.}
Recall that to prove \Cref{lem:per-column}, for a given column $j$, we need to upper bound the probability that more than $\mu/10$ of its at most $k$ rows (with non-zero entries) become dangerous. 
This boils down to understanding the following random process, that we state in a slightly more general form, as we will use it with various parameters settings later.

Let time $t \in [0,n]$ evolve in (tiny) discrete increments of size $dt$, as in the setup of our algorithmic framework. Let $\{Z_t\}_{t\in [0,n]}$ be a random vector-valued process, with increments $d Z_t := Z_{t+d t} - Z_t \in \R^m$, and $Z_0(i)\leq 0$ initially at $t=0$ for each coordinate $i\in[m]$. Assume that $d Z_t$, conditioned on the randomness up to time $t$, is $(\alpha,\theta)$-pairwise independent, for some fixed $\alpha \geq 1$ and $\theta >0$. Let $B \geq 1$ be some target value, and  
call a coordinate $i \in [m]$ {\em bad}  at time $t$ if $Z_t(i) \geq B$.

We wish to bound the number of coordinates that can become bad during the process.
We give such a result below, 
tailored slightly for our purposes (instead of the most general version possible). 

\begin{restatable}[Decoupling bound for $(\alpha,\theta)$-pairwise independent processes]{theorem}{GeneralFreedman}
\label{thm:general_freedman}
Consider the process $\{Z_t\}_{t \in [0,n]}$ above, where $n$ is assumed to be sufficiently large. Suppose that (1) $\alpha \geq 1, \theta > 0, B \geq 1$, and $m$ are all polynomially upper bounded in $n$, (2) $dt  = 1/n^c$ where the constant $c$ can be chosen arbitrarily large, and (3) $\|d Z_t\|_\infty =O(\sqrt{d t})$ for all time steps $t$.\footnote{Here, the constant in $O(\cdot)$ is allowed to depend polynomially on $m, \alpha, \theta$, and $B$, which is at most $\poly(n)$.}

Then for any parameter $\lambda > 0$ satisfying $\lambda \leq \theta/2$ and $\lambda B \leq \log n$,
with probability at least $1 - 1/\poly(n)$, the number of coordinates that ever become bad throughout the process is at most \[
\frac{m}{e^{\lambda B}} + O\big(\frac{\lambda \alpha \log n}{\theta}\big).\] 
\end{restatable}

The proof of \Cref{thm:general_freedman} is somewhat technical and is postponed to \Cref{subsec:proof_general_freedman}.

\subsubsection{Bounding the Number of Dangerous Rows per Column}
\label{subsubsec:dang_row_col_proof_bf_weak}
We can now prove Lemma \ref{lem:per-column}.
Before giving the proof,
let us first revisit the case of $k=\log n$ that we saw informally in Section \ref{subsub:decoupling}, and derive the $\widetilde{O}(\log^{3/4} n)$ bound formally using Theorem \ref{thm:general_freedman}.

\smallskip
\noindent
{\bf Example.} Suppose $k=\log n$, and let $b = o(\log n)$ be the target discrepancy. 

As our target is $b$, we can safely set the threshold $\mu =b$ for tiny rows.
Our goal then is to minimize $b$, while setting the parameters so that for each column $j$, at most $\mu$ medium rows become dangerous (by incurring discrepancy $O(b)$).
So applying \Cref{thm:general_freedman}, we want that
\[ 
\frac{m}{e^{\lambda B}} + O\big(\frac{\lambda \alpha \log n}{\theta}\big) \ll \mu  = b.
\]
Note that $m=k$ and $B=O(b)$ for us. The first term $m/\exp(\lambda B) = k/\exp(\lambda b)$ is easily handled by setting $\lambda  = (\log \log n)/b$. With this $\lambda$, the second term now becomes $\widetilde{O}(\alpha \log n)/ \theta b$.

Recall that $Y_t$ is $(\alpha,\theta)$-pairwise independent with $\alpha \approx k/\mu = k/b$ and $\theta \approx \beta \approx b/k$. Plugging these, the second term becomes $\widetilde{O}(k^2 \log n)/  b^3$, which we want to keep below $\mu = b$, i.e.,~ensure that
\[ \widetilde{O}(k^2 \log n)/  b^3 \ll b.\]
This holds when $b = \widetilde{O}(\log^{3/4} n)$  in the case of $k=\log n$.

We now prove Lemma \ref{lem:per-column}.

\begin{proof}
Recall that a row $i$ is dangerous at time $t$, if (1) it is medium at time $t$, and (2) its discrepancy satisfies $\langle a_i,x_t\rangle \geq 2 b$. 

As the regularized discrepancy $Y_t(i)$ of row $i$ always upper bounds its discrepancy, and $Y_{t_i^*}(i)\leq b$ where $t_i^*$ is the time when the row first becomes medium, for row $i$ to become dangerous its regularized discrepancy must increase by at least $b$. 

To prove the lemma, fix some column $j \in [n]$, and consider the vector $Y_t$ of regularized discrepancies restricted to coordinates in $\mathcal{C}_j$ (i.e., rows $i$ with $a_i(j)\neq 0$). We  denote this vector $Y_{t,j}$ for short. 
By \Cref{lem:weak_pairwise_bf_weak}, the increments of $Y_{t,j}$ is $(4k/\mu, \beta/5)$-pairwise independent at any time $t$.

We now apply \Cref{thm:general_freedman} to $Z_t = Y_{t,j} - b$ with target $B=b$.\footnote{Rigorously speaking, $Y_{t,j}$ is only defined for the set of medium rows in $\mathcal{C}_j$, which changes over time. However, we can easily bypass this technical nuance by additionally defining $Y_{t,j}(i) := b$ for large rows $i$, and $-\infty$ for tiny rows. This way $Y_{t,j}(i)$ only decreases (discontinuously) when it changes from large to medium, and from medium to tiny.} Then the dimension parameter $m \leq k$ as $|\mathcal{C}_j| \leq k$, and we have  $\alpha = 4k/\mu$, $\theta = \beta/5 = b / (50k)$.

So the bound in \Cref{thm:general_freedman}  becomes
\begin{align} 
\frac{k}{e^{\lambda b}} + O\Big(\frac{\lambda \alpha \log n}{\theta}\Big)  
&=  \frac{k}{e^{\lambda b}} + O\Big(\frac{\lambda k^2 \log n}{\mu b}\Big)
\nonumber.
\end{align}
We now set $\lambda = (\log\log n)/b$. Notice that this satisfies the condition  $\lambda \leq \min(\theta/2,1)$ in Theorem \ref{thm:general_freedman}, as $\theta = \beta/5 = O(b/k) \gg (\log\log n)/b$ by our assumption (ii) in the lemma   that $b \gg (k \log\log n)^{1/2}$.
Plugging this $\lambda$ in the bound above gives,
\begin{align} 
 \frac{k}{e^{\lambda b}} + O\Big(\frac{\lambda k^2 \log n}{\mu b}\Big)
 \leq \frac{k}{\log n} + O \Big(\frac{k^2 \log n \log \log n}{\mu b^2}\Big) \ll \mu \label{eq:bad_rows_bf_weak_2} .
\end{align}
The last inequality uses two facts.
First, we have that $ k/\log n \ll \mu$,  as $\mu \geq b^2/\log n$ (by definition of $\mu$), and as we always have $b\gg  k^{1/2}$.
Second, we have that $k^2 \log n \log \log n \ll b^2 \mu^2$, by 
our assumption (i) in the lemma  that $b \mu \gg k (\log n \log \log n)^{1/2}$.

Thus Theorem \ref{thm:general_freedman} implies  
that with probability at least $1 - 1/\poly(n)$, the number of dangerous rows in $\mathcal{C}_j$ at any time $t$ is at most $\mu/10$.
Taking a union bound over all columns $j$ and all time steps $t$ gives the result.
\end{proof}

\subsubsection{Proof of the Decoupling Bound}

\label{subsec:proof_general_freedman}

Finally, we prove the decoupling bound in \Cref{thm:general_freedman}, which is restated below.

\GeneralFreedman*

\begin{proof}
Consider the moment generating function $\Phi_t(i) := 
    \exp (\lambda Z_t(i))$ of $Z_t(i)$ with parameter $\lambda$. 

Using $|\lambda d Z_t(i)| \leq 1$ and $\exp(x) \leq 1+ x +x^2 $ for $|x| \leq 1$,
the change in $\Phi_i(t)$ is bounded by  
\begin{align}
\label{eq:change_MGF_freedman}
d \Phi_t(i) = \Phi_t(i) ( \exp(\lambda d Z_t(i)) -1) \leq \gamma_t(i) \cdot (d Z_t(i) + \lambda (d Z_t(i))^2) , 
\end{align}
where we define $\gamma_t(i) := \lambda \Phi_t(i)$.

To prove the theorem, we use the following proxy for the number of bad coordinates
\[
W_t := \sum_{i \in [m]} \min\{\Phi_t(i), e^{\lambda B}\}  =: \sum_{i \in [m]}\trunc_B(\Phi_t(i)) .
\]
Note that $W_0 \leq m$, as each $Z_0(i) \leq 0$, and at any time $t$, the number of bad coordinates is at most $W_t / e^{\lambda B}$. 
We will bound how much $W_t$ can exceed beyond $W_0$, by apply the Freedman-type inequality in \Cref{fact:freedman_conc_komlos}. For this, we first upper bound $\E[d W_t]$ and $\E[(d W_t)^2]$ suitably.

Notice that for each bad coordinate $i$ at time $t$ (where $Z_t(i)\geq B$), the quantity $\trunc_B(\Phi_t(i))$ cannot increase in the next step. 
Call a coordinate {\em good} if it is not bad, and 
let $\mathcal{G}_t$ denote the set of good coordinates at time $t$. Notice also that 
for each good coordinate $i \in \mathcal{G}_t$, due to the truncation, the change $d \trunc_B(\Phi_t(i))$ satisfies $d \trunc_B(\Phi_t(i)) \leq d \Phi_t(i)$. 

Using the above observations together with \eqref{eq:change_MGF_freedman}, we have that 
\begin{align}
\label{eq:delta-wt}
d W_t \leq \sum_{i \in \mathcal{G}_t} \gamma_t(i) \cdot (d Z_t(i) + \lambda (d Z_t(i))^2).
\end{align}
Thus the (conditional) first moment $\E_t[d W_t]$ of $d W_t$ satisfies
\begin{align} 
\label{eq:truncMGF_drift}
\E_t [d W_t] 
& \leq \sum_{i \in \mathcal{G}_t} \gamma_t(i) \cdot (\E_t[d Z_t(i)] + \lambda \cdot \E_t[(d Z_t(i))^2] ) \nonumber \\
& \leq \sum_{i \in \mathcal{G}_t} \gamma_t(i) \cdot ( - \theta \cdot \E_t[(d Z_t(i))^2] + \lambda \cdot \E_t[(d Z_t(i))^2]) \nonumber \\
& \leq - \sum_{i \in \mathcal{G}_t} \gamma_t(i) \cdot \frac{\theta}{2} \cdot \E_t[(d Z_t(i))^2] \qquad \qquad \text{(as $\lambda \leq \theta/2$)},
\end{align}
where the second inequality uses that $d Z_t$ is (conditionally) $(\alpha,\theta)$-pairwise independent.

Next we can upper bound the (conditional) second moment  $\E_t[(d W_t)^2]$ as \footnote{Here we can ignore the lower order terms arising from $\gamma_t(i) \lambda (d Z_t(i))^2$ in \eqref{eq:delta-wt}. As we have assumed that $\lambda B \leq \log n$ and $B \geq 1$, then $\gamma_t(i) \lambda = \lambda^2 \exp(\lambda B) \leq n \log^2 n$, and as $(d Z_t(i))^2 = \poly(n) d t$ by our assumptions, we have  $\sum_{i \in \mathcal{G}_t} \gamma_t(i)  \lambda (d Z_t(i))^2 = \poly(n) dt$, and thus the contributions involving the term $\gamma_t(i) \lambda (d Z_t(i))^2$ in $(d W_t)^2$ are at most $O(\poly(n) (dt)^{3/2})$. Since $dt = 1/n^c$ for an arbitrarily large constant $c$, these extra terms contribute negligibly to $W_t$ over all the time steps $t \in [0,n]$.}
\begin{align}
\label{eq:trunc-freedman}
\E_t[(d W_t)^2] 
& \leq  \E_t \Big( \sum_{i \in \mathcal{G}_t} \gamma_t(i) d Z_t(i) \Big)^2 
 \leq 2 \alpha \cdot \sum_{i \in \mathcal{G}_t} \gamma_t(i)^2 \cdot \E_t[(d Z_t(i))^2] \nonumber \\
 & \leq  2 \alpha \lambda \exp(\lambda B)   \cdot \sum_{i \in \mathcal{G}_t} \gamma_t(i) \cdot \E_t[(d Z_t(i))^2], 
\end{align}
where the second inequality uses the conditional $(\alpha,\theta)$-pairwise independence of $d Z_t$, and the third inequality follows $\gamma_t(i) \leq \lambda \exp(\lambda B)$ for all good coordinates $i$.

\noindent \textbf{Applying Freedman-Type Inequality.}
Together \eqref{eq:truncMGF_drift} and \eqref{eq:trunc-freedman} give that $W_t$ satisfies  \[ \E_t [d W_t] \leq - \delta \E_t [(d W_t)^2]\]
with $\delta := \theta/(4 \alpha \lambda \exp(\lambda B))$.
Thus, applying \Cref{fact:freedman_conc_komlos} 
 with  $\xi := O(\delta^{-1}   \log n)$ (here we take the natural map of time $t \in [0,n]$ to $0, \cdots, T = n / dt$), 
  we have that for any time step $t \in [0,n]$,
\[ \Pr[ W_t - W_0 \geq \xi] \leq \exp(-\delta  \xi) =  \exp(-O(\log n)).\]
Choosing the constant $O(\cdot)$ in $\xi$ large enough, and applying a union bound over all time steps $t \in [0,n]$, and using that $W_0\leq m$, gives that $W_t \leq m + \xi$ for all $t \in [0,n]$ with probability at least $1 - 1/\poly(n)$. 
The theorem now follows as each bad coordinate contributes $e^{\lambda B}$ to $W_t$, and their number can be at most \[\frac{m + \xi}{ e^{\lambda B}} = \frac{m}{e^{\lambda B}} + O\Big(\frac{\lambda \alpha \log n}{\theta}\Big).\qedhere \]  
\end{proof}

%% file: Komlos_Proof.tex
\section{A Better Bound for the Koml\'os Problem}
\label{sec:komlos}

Let $A \in \mathbb{R}^{m \times n}$ be a  matrix whose columns have $\ell_2$ norm at most $1$. 
We now use the ideas from \Cref{sec:bf_weak} to give an algorithm that achieves $\disc(A) = \widetilde{O}(\log^{1/4}n)$, hence proving Theorem \ref{thm:Komlos}.

To see where the $\log^{1/4} n$ comes from,
consider the case when  all non-zero entries in $A$ have equal magnitude, i.e., $|a_i(j)|=v$. 
As this is simply a Beck-Fiala instance (scaled by $v$) with column-sparsity $k \leq 1/v^2$, applying Theorem \ref{thm:bf_weak-informal}, which gives an  $\widetilde{O}(k^{1/2} \log^{1/4}n)$ bound for the Beck-Fiala problem for every $k$, we get an $ \widetilde{O}(vk^{1/2} \log^{1/4}n) = \widetilde{O}(\log^{1/4}n)$ bound for every $v$.

To handle general instances, we will split the entries of $A$ into different scales based on their magnitude, and run the Beck-Fiala algorithm ``in parallel" on these scales. 
However, as the magnitudes $|a_i(j)|$ can range from $1/n$ to $1$, doing this naively would result in $O(\log n)$ scales, leading to an unacceptable $O(\log n)$ factor loss. 

Instead, we show that it suffices to consider $O(\log \log n)$ scales --- one each for the {\em heavy} entries $ |a_i(j)| \in  (1/\log^5 n,1]$, and another single special {\em light} scale for all $|a_i(j)| \leq 1/\log^5 n$. 
For the heavy scales, we can essentially apply the Beck-Fiala algorithm from \Cref{sec:bf_weak} in a black-box way.
To handle the special light scale however, we will need to open the black box and slightly modify the algorithm and its analysis.\footnote{To keep the exposition modular and as simple as possible, we do not attempt to optimize the $\log \log n$ factors.} We now give the details.

\subsection{A Multi-Scale Algorithm}
\label{subsec:alg_komlos}

Call an entry $a_i(j)$ {\em heavy} if $|a_i(j)| >  1/\log^5 n$, and {\em light} otherwise.
We split the entries into $P$ scales, where $P := 1+ 5 \log \log n$. Scales $p=1,\ldots,P-1$ will consist of heavy  entries, where $a_i(j)$ lies in scale $p$ if $|a_i(j)|  \in (2^{-p},2^{-p+1}]$. 
All the light entries lie in the single scale $P$. Note that unlike the $P-1$ heavy scales, the magnitude of the entries in the light scale $P$ can vary widely.

For each row $a_i$ of $A$, we create $P$ rows $a_i^{(1)},\ldots,a_i^{(P)}$, where each row  $a_i^{(p)}$ consists of scale-$p$ entries of $a_i$, scaled by $2^{p-1}$, i.e., for scales $p \leq P-1$, $a_i^{(p)}(j) := 2^{p-1} a_i(j)$ if $|a_i(j)| \in (2^{-p} ,2^{-p+1}]$, and $0$ otherwise; for the light scale $P$, $a_i^{(P)}(j) := 2^{P-1} a_i(j)$ if $|a_i^{(P)}(j)| \leq 1/\log^5 n$, and $0$ otherwise. 
For each $p\in [P]$, let $A^{(p)}$ denote the matrix consisting of rows $a_i^{(p)}$. Note that $A = \sum_{p \leq P} 2^{-p+1} A^{(p)}$.

For $p \leq P-1$, each $A^{(p)}$ is (essentially) a Beck-Fiala instance with up to $k_p := 2^{2p}$ non-zero entries per column with magnitudes in $(1/2,1]$.
For the light scale $P$, the column sparsity of $A^{(P)}$
can be much larger than $2^{2P}$,
but as each column of $A$ has $\ell_2$-norm at most $1$, we have  $\sum_i (a_i^{(P)}(j))^2 \leq 2^{2P}$ for each column $j$ of $A^{(P)}$.
So we will abuse the notation and still denote $k_{P} := 2^{2P}$ (but the light scale $P$ will be analyzed separately from the others).

\smallskip
\noindent \textbf{Overview.}
As stated previously, we will run the algorithm in \Cref{sec:bf_weak} for all the instances $A^{(p)}$ with $p\in [P]$ ``in parallel''. This means that we define the subspaces $W_t^{(p)}$, the matrices $E_s^{(p)}$ and the associated SDP constraints separately for each $A^{(p)}$, but solve a single SDP with all these constraints together, to find the coloring update $v_t$.
This will allow us to simultaneously control the discrepancy of each $A^{(p)}$, and give that  $\disc(A) \leq \sum_{p\leq P} 2^{-p+1} \disc(A^{(p)})$.

More precisely,
let $B(k)$ denote the discrepancy in Theorem \ref{thm:bf_weak}  for Beck-Fiala instances with column sparsity $k$. Recall that $B(k)
= \widetilde{O}(k^{1/2} (\log n)^{1/4})$ for all $k$.
For $p\leq P-1$, running this algorithm  only on $A^{(p)}$ incurs discrepancy  $B(k_p)$. The scale $p=P$ needs more care, but using the ideas in Section \ref{sec:bf_weak}, we will also be able to show that $\disc(A^{(P)}) = O(B(k_{P}))$.
Moreover, running the algorithm simultaneously on all the $A^{(p)}$ in parallel will only incur an extra $O(P^{1/2})$ loss on top of these individual bounds.
Together, these will give that 
\begin{equation}
    \label{eq:komlos-bound}
\disc(A) \leq \sum_{p\leq P} 2^{-p+1} \cdot  B(k_p) \cdot O(P^{1/2}) \leq   \widetilde{O}(\log^{1/4} n) \cdot P^{3/2} = \widetilde{O}(\log^{1/4}n).
\end{equation}

We now give the details. We start with some relevant definitions, analogous to those in \Cref{subsec:alg_bf_weak}.

\smallskip
\noindent \textbf{Large and Tiny Rows.} At time $t$, for each scale $p \in  [P]$, we say that a row $i$ is {\em scale-$p$ large} if  $\sum_{j \in \mathcal{V}_t} (a_i^{(p)}(j))^2 > 10 k_p$. 
As the matrix $A$ has column $\ell_2$-norm at most $1$, it is easily verified that    at any time $t$, the total number of large rows (across all scales) is at most $n_t/10$. Large rows will never incur any discrepancy as we will walk orthogonal to them.

For each scale $p \in [P]$, we define $b_p := B(k_p) P^{1/2}$. The target discrepancy for $A^{(p)}$ will be $O(b_p)$. 

We call a row $i$  {\em scale-$p$ tiny} at time $t$ if 
$\sum_{j \in \mathcal{V}_t} (a_i^{(p)}(j))^2\leq \mu_p$, where $\mu_p := \max(b_p, b_p^2/\log n)$.
As previously,  with high probability, the discrepancy incurred by any row will be $O(b_p)$ after it becomes scale-$p$ tiny.

\smallskip
\noindent \textbf{Medium Rows and Scale-$p$ Matrix $E^{(p)}_t$.}
At time $t$, a scale-$p$ row that is neither large nor tiny is called {\em scale-$p$ medium}. We will mostly focus on such rows, and denote the set of such rows by $\mathcal{M}_t^{(p)}$. 
Let $\beta_p := b_p/(10k_p)$. 
For each  time $t$ and scale $p$, let $E^{(p)}_t \in \mathbb{R}^{|\mathcal{M}_t^{(p)}| \times n_t}$ be the matrix whose rows $E^{(p)}_{t,i}$ correspond to scale-$p$ medium rows $i \in \mathcal{M}_t^{(p)}$ with entries $E^{(p)}_{t,i}(j) := a_i^{(p)}(j) - 2 \beta_p (a_i^{(p)}(j))^2 x_t(j)$, restricted to the alive elements $j \in \mathcal{V}_t$. 

\smallskip
\noindent {\bf Dangerous Rows.} For each $p\in [P]$, call a row $i$ of $A^{(p)}$ {\em dangerous} at time $t$, if (i) it is scale-$p$ medium, and (ii) its current discrepancy $\langle a^{(p)}_i, x_t \rangle \geq 2b_p$. 

\smallskip
We can now the state the algorithm formally.

\smallskip
\noindent \textbf{The Algorithm.} Consider the matrices $A^{(p)}$ and the quantities defined above.
At each time $t$, repeat the following until $n_t \leq  10 P$.
\begin{enumerate}
    \item  Let $W_t$ be the subspace spanned by 
(i) for each $p\in [P]$, the scale-$p$ large rows $a_i^{(p)}$, (ii) rows $E^{(p)}_{t,i}$ for all dangerous rows $i$ in $A^{(p)}$, (iii) the vector $x_t$. If $\dim(W_t) > n_t/3$, declare FAIL. 
    \item Solve the SDP \eqref{sdp:jj_new}-\eqref{sdp:psd_new} with the subspace $W = W_t$, the matrices $E_p = E^{(p)}_t$ for $p \in [P]$, with parameters $\kappa = 1/6$, $\eta= \eta_{P} = 1/6$ and $\eta_p =1/(6P)$ for each $p \in [P-1]$. 
    Use the resulting SDP solution $U_t$ to find $v_t$ as in \eqref{eq:find_vt}, and update the coloring $x_t$ by $v_t \sqrt{dt}$.
\end{enumerate}

\subsection{The Analysis}
We show that with high probability, the algorithm outputs a coloring with $\disc(A) = \widetilde{O}(\log^{1/4}n)$.  

\begin{proof}[Proof of \Cref{thm:Komlos}]
The proof is very similar to that in \Cref{sec:bf_weak}, with two main differences --- the parameters $\eta_p$ are slightly different, and the light scale $P$ needs to be handled differently due to varying magnitudes of entries.

First, note that if the algorithm does not FAIL in Step 1 at time $t$, then $\dim(W_t)\leq n_t/3$ and hence $\delta\leq 1/3$. So by \Cref{thm:sdp_feasibility}, the SDP in Step 2 is feasible as our parameters satisfy $\delta + \kappa + \eta + \sum_p \eta_p \leq 1$.

So it suffices to show that, with high probability, (i) the algorithm never FAILs in Step 1, and (ii) that $\disc(A^{(p)}) =O(b_p)$ for each scale $p\leq P$.
The claimed bound on $\disc(A)$ then follows by \eqref{eq:komlos-bound}.

\smallskip
{\bf Bounding the Discrepancy at Each Scale.} 
As usual, each scale-$p$ row incurs zero discrepancy while it is scale-$p$ large, as they are blocked by our choice of $W_t$. 
As usual, we can also ignore scale-$p$ tiny rows, since $\eta=1/6$ and thus SDP satisfies $U_t \leq 6 \cdot \diag(U_t)$ at all $t$.  
So we only need to consider medium rows, and as before, we will work with the regularized scale-$p$ discrepancy, 
\begin{align} \label{eq:class_p_reg_disc}
Y_t^{(p)}(i)  := 
    \langle a_i^{(p)}, x_t \rangle  + \beta_p  \sum_{j=1}^n a_i^{(p)}(j)^2 (1 - x_t(j)^2) ,
\end{align}
where recall that $\beta_p = b_p/(10k_p)$ for scale $p$.

If the algorithm never FAILs in Step 1, then, as previously,
the discrepancy of every scale-$p$ medium row stays $O(b_p)$. This is because the
regularized discrepancy of any scale-$p$ row $i$  can only decrease (deterministically) once it becomes dangerous, as we
walk orthogonal to $E^{(p)}_{t,i}$  for each such row.

\smallskip
{\bf Bounding the Number of Dangerous Rows.}
We will show that, with high probability, the total number of dangerous rows never exceeds $n_t/10$ at any time $t$, and thus the algorithm never declares FAIL.
To this end, analogous to Lemma \ref{lem:per-column}, it suffices to show that, with high probability:

\,\,\,\,\,\,\,$(*)$ \,\,\,
For each scale $p\in [P]$, every column $j$ has at most $\mu_p/20P$ dangerous rows from scale $p$.

\smallskip
The rest of the proof is devoted to proving $(*)$.
Notice first that we can ignore all scales $p$ with $\beta_p>0.1$.
Indeed, consider some row $i$ in  scale-$p$ with $\beta_p>0.1$, and let $t_i$ be the first time it is not large. Then $\sum_{j \in \mathcal{V}_{t_i}} (a_i^{(p)}(j))^2 \leq 10 k_p$, and as each non-entry satisfies $|a_i^{(p)}(j)| \in [1/2,1]$, the discrepancy can be at most $O(\sum_{j \in \mathcal{V}_{t_i}} |a_i^{(p)}(j)|) = O(k_p)$, which is $O(b_p)$ as $\beta_p = b_p/(10k_p) \geq 0.1$. 

\smallskip
{\bf Dangerous Rows from Heavy Scales $p\leq P-1$.} We assume henceforth that $\beta_p \leq 0.1$.
For such heavy scales, the analysis is essentially identical to that in \Cref{sec:bf_weak}.  
The only difference now is that (i) we have $\eta_p = \eta /P$, and (ii) we need a tighter bound of $O(\mu_p/P)$ on dangerous rows per column (instead of $O(\mu_p)$). Both of these differences are handled by the fact we increase our target discrepancy $b_p$ by the $P^{1/2}$ factor.

Let us elaborate some more. Analogous to \Cref{lem:weak_pairwise_bf_weak}, we have that the increment of $Y_t^{(p)}$ is $(\alpha_p, \theta_p)$-pairwise independent, where $\alpha_p = k_p/(\mu_p\eta_p)$ and $\theta_p = \beta_p/5 = O(b_p/k_p)$. 
By Theorem \ref{thm:general_freedman}, we wish to show that our target $b_p$ satisfies \[k_p/\exp(\lambda_p b_p) + O(\lambda_p \alpha_p \log n)/ \theta_p \ll  \mu_p/P .
\]
Setting $\lambda_p = O(\log \log n)/b_p$ (as before), the first term becomes negligible as $\mu_p \geq b_p^2/\log n \geq k_p/\log n$. 
For the second term, we need to show that
\[
\lambda_p \alpha_p \log n \ll  \theta_p \mu_p/P.
\]
The only difference now from before, is that $\eta_p = \eta/P$ is a factor $P$ smaller and $b_p$ is factor $P^{1/2}$ larger (than the corresponding $B(k_p)$). Note that increasing $b_p$ also increases $\mu_p$ by at least a $P^{1/2}$ factor.
As $\alpha_p = k_p/(\mu_p\eta_p)$, the effect of $P$ in the left term $\lambda_p \alpha_p$ cancels out. On the right side, $\theta_p = \beta_p/5 = O(b_p/k_p)$ increases by factor $P^{1/2}$ and thus $\theta_p \mu_p$ increases by factor $\Omega(P)$, which offsets the extra factor $P$ in the denominator. This shows that $(*)$ holds for all heavy scales $p \leq P-1$.

\smallskip
\noindent \textbf{Dangerous Rows from the Light Scale $p=P$.}   
Recall that the non-zero entries in $A^{(P)}$ satisfy $|a_i^{(P)}(j)|\in (1/n,1)$, and does not directly correspond to  a Beck-Fiala instance.
To handle this, we will further partition these entries into $O(\log n)$ subscales, and then apply \Cref{thm:general_freedman} suitably to each subscale. 

We say that an entry $a_i^{(P)}(j)$ is in subscale $q \in [\log n]$,  if 
$|a_i^{(P)}(j)| \in (2^{-q},2^{-q + 1}]$.  
For any column $j$, 
let $k_{j,q}^{(P)}$ denote the number of scale-$P$ subscale-$q$ entries. 
As $\sum_i (a_i^{(P)}(j))^2 \leq k_P$, we have that $\sum_{q} 2^{-2q} k_{j,q}^{(P)} \leq  k_P$.

As in \Cref{lem:weak_pairwise_bf_weak}, the increment of $Y_t^{(P)}$ is
$(\alpha,\theta)$-pairwise independent with $\alpha = k_{P}/(\mu_P\eta_P)  \leq  6 \log n$, 
as $\mu_P = b_P^2/\log n \geq k_P/\log n$ and $\eta_P=1/6$, and with
 $\theta = \beta_P/10 = \Omega(b_P/k_{P})$.
Then $Y_{t,j,q}^{(P)}$, the restriction of $Y_t^{(P)}$ to subscale-$q$ coordinates in column $j$, is also $(\alpha,\theta)$-pairwise independent.

Applying \Cref{thm:general_freedman} to $Y_{t,j,q}^{(P)}$ with $\lambda = (5 \log\log n)/B$ and $B = b_P$, we obtain that, with high probability, the number of dangerous subscale-$q$ coordinates in column $j$ is at most
\[
\frac{k_{j,q}^{(P)}}{\log^5 n} + O\Big(\frac{\lambda \alpha \log n}{\theta}\Big) \leq \frac{k_{j,q}^{(P)} }{\log^5 n} + 
O\Big(\frac{ k_P \log^2 n \log \log n}{b_P^2} \Big) = 
\frac{k_{j,q}^{(P)}}{\log^5 n} + O(\log^2 n),
\]
where the last step uses that $b_P =\Omega(\sqrt{k_P  \log \log n})$.

As each subscale-$q$ coordinate $i$ satisfies $|a^{(P)}_i(j)|^2 \in (2^{-2q}, 2^{-2q + 2}]$, this implies that with high probability, the squared $\ell_2$-norm of all the dangerous coordinates in column $j$ is at most 
\[ \sum_q  2^{-2q+2} \Big( \frac{k_{j,q}^{(P)}}{\log^5 n} + O(\log^2 n)  \Big) \leq  
\frac{4k_{P}}{\log^5 n} + O(\log^2 n) \ll \frac{k_{P}}{\log^2 n}  \ll \frac{\mu_P}{P},
\]
where the second inequality uses $k_{P} = 2^{2P} = \log^{10} n$, and the last step uses that $\mu_P \geq k_P/\log n$. 
This proves $(*)$ for the light scale $P$. 

Together, these imply that with high probability, at each time $t$, at most $n_t/10$ rows can become dangerous and hence the algorithm does not FAIL.
\end{proof}

%% file: Large_k_Beck_Fiala_Conjecture.tex
\section{Resolving the Beck-Fiala Conjecture for $k = \Omega(\log^2 n)$}
\label{sec:bf_better}

We now describe the multi-layered extension of our basic method and use it to prove Theorem \ref{thm:bf_conj_large_k}. 
For completeness, we prove the following statement which lists the bounds for all regimes of $k$.

\begin{restatable}[Stronger Beck-Fiala bound]{theorem}{BeckFialaBetter}\label{thm:bf_better}
Let $A \in \{0, \pm 1\}^{m \times n}$ be a matrix where each column has at most $k$ non-zeros. 
Then there is a polynomial time algorithm that finds a coloring with
\[\disc(A) =
\begin{cases}
  O(k)   & \text{ if $k \leq  \log^{1/2} n$}\\
  O(k^{1/3} \log^{1/3}n)  & \text{ if $  \log^{1/2}n \leq k \leq \log^2 n$}\\
  O(k^{1/2}) & \text{ if $k \geq \log^{2}n$}. 
\end{cases}
\]
\end{restatable}

Notice that there are no hidden $\poly(\log \log n)$ factors in these bounds any more.
In particular, for $k \geq \log^2 n$,  this exactly resolves the Beck-Fiala conjecture for such $k$, proving Theorem \ref{thm:bf_conj_large_k}.

This also improves upon the basic bounds in Theorem \ref{thm:bf_weak-informal} in other ways. In particular, for $k= \log n$, the discrepancy above is $O(\log^{2/3} n)$, as opposed to $\widetilde{O}(\log^{3/4}n)$ previously.

\smallskip

\noindent \textbf{High-Level Idea.} 
As discussed previously, instead of directly bounding the number of dangerous rows per column by $O(\mu)$ (recall that $\mu$ is the size threshold for tiny rows) in one shot as in \Cref{lem:per-column}, the idea is to adopt a multi-layered strategy --- where at each level, we only aim at reducing the column sparsity of the dangerous rows in any column by a {\em constant factor}. 

This allows us to remove of the $\log \log n$ factor in $\lambda$ in the proof of \Cref{lem:per-column},
which removes a $\log \log n$ factor in \Cref{lem:per-column} and thus \Cref{thm:bf_weak}.

For clarity of presentation, we first describe the algorithm and its analysis for $k = \Omega(\log^2 n)$, where the discrepancy bound in \Cref{thm:bf_better} is $O(\sqrt{k})$. We consider the case of $k \ll \log^2 n$ in Section \ref{subsec:small_k_bf_better}, which follows the same ideas, but requires setting some parameters differently.

\subsection{A Multi-Layered Algorithm for $k = \Omega(\log^2 n)$} 
\label{subsec:alg_bf_large_k}

Let $b_0 := C \sqrt{k}$ be the target discrepancy bound, where $C$ is some large enough constant. 
As usual, we can ignore all tiny or irrelevant rows $i$ with sizes at most $\mu := \max(b_0,b_0^2/\log n)$, as they incur at most $O(b_0)$ additional discrepancy, with high probability. As we are in the regime of $k = \Omega(\log^2n$), the target discrepancy $b_0 = \Omega(\log n)$, and thus $\mu= b_0^2/\log n$ here.

\noindent \textbf{Level of Non-Tiny Rows.} For each time $t$, we assign a (unique) level $\ell_t(i) \in \{0,1, \cdots, L\}$ to every row $i$ that is not tiny, where we choose $L:= \log_{100}(10 k/\mu) = O(\log \log n)$.
It is useful to think of these as ``danger" levels.
At time $0$, every row $i$ starts at level $\ell_0(i) = 0$, and its level can only increase over time, depending on how much discrepancy it accrues.

Our algorithm will guarantee (with high probability) that at each time $t$, each column $j$ has sparsity at most  $k_\ell := k/(100^\ell)$ when restricted to level-$\ell$ rows. 
In other words, the set of level-$\ell$ rows corresponds to a Beck-Fiala instance with sparsity $k_\ell \ll k$.
Our goal will be to bound the total discrepancy incurred by a row $i$ during the time it is at level $\ell$ by $O(b_{\ell})$, where 
\[b_\ell := C 2^\ell \sqrt{k_\ell} = C 5^{-\ell} \sqrt{k} = 5^{-\ell}b_0.\] 
This would imply \Cref{thm:bf_conj_large_k}, as the total discrepancy incurred for a row over all possible levels is at most $O(\sum_{\ell \geq 0} b_\ell) = O(\sqrt{k})$. Here, note that the $100^\ell$ in the denominator of $k_\ell$ offsets the $2^\ell$ factor in the setting of $b_\ell$, which is crucial for avoiding extra $\log\log n$ factors in the bound.

The level of a row $i$ is defined inductively. Let $t^{(\ell)}_i$ denote the first time when row $i$ reaches level $\ell$. For $\ell = 0$, we set $t^{(0)}_i = 0$ for all rows that are not tiny at time $0$. We now specify when a row $i$ at level $\ell$ gets {\em upgraded} to
level $\ell + 1$. 

\smallskip
{\bf Large and Medium Rows, and Level Upgrades.}
At time $t$, call a level-$\ell$ row $i$ {\em large}, if its size is more than $2^\ell \cdot 10 k_\ell$.
 A level-$\ell$ row is called {\em medium} or {\em interesting} if it is neither large nor tiny.
 
For each medium level-$\ell$ row $i$ at time $t$, we define its regularized discrepancy
\begin{align} \label{eq:reg_disc_bf_large_k}
Y_t^{(\ell)}(i) := \langle a_i, x_t - x_{t^{(\ell)}_i} \rangle + \beta_\ell G_t(i) ,
\end{align}
where $\beta_\ell := b_\ell/(2^\ell \cdot 10 k_\ell)$ and $G_t(i) = \sum_{j \in \mathcal{V}_t} a_i(j)^2 (1 - x_t(j)^2)$ is the energy as in \Cref{sec:bf_weak}. 

Note that the second term in \eqref{eq:reg_disc_bf_large_k} is at most $\beta_\ell 2^\ell \cdot 10 k_\ell \leq b_\ell$, and the first term in \eqref{eq:reg_disc_bf_large_k} measures the discrepancy row $i$ incurs during its time at level $\ell$. 

We say that a level-$\ell$ row $i$ becomes dangerous if $Y_t^{(\ell)}(i) \geq 2 b_\ell$, and the level of $i$ is immediately upgraded to $\ell+1$ as soon as it becomes dangerous at level $\ell$. 
In other words, we set $t_i^{(\ell+1)}$ to be the first time that the level-$\ell$ row $i$ becomes dangerous.
Dangerous level-$L$ rows do not get upgraded --- they will simply be blocked.

We denote the level-$\ell$ medium rows by $\mathcal{M}_t^{(\ell)}$. 
Let $E_t^{(\ell)} \in \mathbb{R}^{|\mathcal{M}_t^{(\ell)}| \times n_t}$ denote the matrix whose $i$th row is $a_i - 2 \beta_\ell a_i^2 x_t$ restricted to the alive elements $\mathcal{V}_t$, for all medium level-$\ell$ rows $i$.

\smallskip

We can now describe the algorithm.

\smallskip
\noindent \textbf{The Algorithm.} We run the algorithm in \Cref{sec:bf_weak} for all the levels in parallel. 
This is done by choosing the subspace $W_t$ and the matrices $E_s$ in the SDP at time $t$ as follows. 

\begin{enumerate}
    \item Let $W_t$ denote the subspace spanned by (i) all the large rows $a_i$ at all levels, (ii) the $i$th row of $E^{(L)}_t$ for all dangerous level-$L$ rows $i$, (iii) the vector $x_t$. If $\dim(W_t) > n_t/3$, declare FAIL. 

    \item Solve the SDP \eqref{sdp:jj_new}-\eqref{sdp:psd_new} with the subspace $W = W_t$, the matrix $E_\ell = E_t^{(\ell)}$ and parameters  $\delta=1/3, \kappa = 1/6$, $\eta = 1/6$, and $\eta_\ell = \eta/2^\ell$ for all levels $\ell \in \{0, \cdots, L\}$.
    Use the resulting SDP solution $U_t$ to find $v_t$ as in \eqref{eq:find_vt}, and update the coloring $x_t$ by $v_t \sqrt{dt}$.
\end{enumerate}

\subsection{The Analysis for $k = \Omega(\log^2 n)$}
We now analyze the algorithm.
The key to the analysis is to show that within each column $j$, with high probability, for all levels $\ell \in \{0,\ldots,L\}$, at most $k_\ell/100$ rows become dangerous and get upgraded to level $\ell+1$ (and for $\ell=L$, they are blocked). 

\begin{lemma}[Number of dangerous rows]
\label{lem:num_dang_bf_large_k}
Assume $k = \Omega(\log^2 n)$. Then, with high probability, for all columns $j \in [n]$ and all levels $\ell \in \{0,\cdots,L\}$, at most $k_\ell$ rows in column $j$ ever reach level $\ell$. 
Moreover, at most $k_L/100$ rows in column $j$ ever become dangerous at level $L$.
\end{lemma}

\begin{proof}
We prove the lemma by induction on the level $\ell$. The base case of $\ell = 0$ is trivial as $k_0=k$. 

Suppose that the lemma holds for some level $\ell \geq 0$. Fix some column $j$. We will prove that with high probability
at most $k_\ell/(100)$ level-$\ell$ rows in column $j$ ever become dangerous throughout the algorithm. As $k_{\ell+1}=k_\ell/100$, this will show that the statement also holds for level $\ell+1$. 
For $\ell = L$, this will give the ``moreover'' part of the statement.

Analogous to the proof of \Cref{lem:per-column}, we will apply \Cref{thm:general_freedman} to the regularized discrepancy vector $Y_t^{(\ell)}$ for each level $\ell$. 
First, we address a technical point. Note that the coordinates on which $Y_t^{(\ell)}$ is defined might change over time, either when the sizes of level-$\ell$ rows changes (from large to medium or from medium to tiny), or when rows get upgraded (to and from level $\ell$). 
However, this issue is easily addressed by extending the definition of $Y_t^{(\ell)}(i)$ to all rows $i$, where $Y_t^{(\ell)}(i) := b_\ell$ if row $i$ is either level-$\ell$ large, or has not reached level $\ell$ yet, and $Y_t^{(\ell)}(i) := -\infty$ if row $i$ becomes tiny.  
This guarantees that $Y_t^{(\ell)}(i)$ only decreases (discontinuously) when row $i$ enters or leaves $\mathcal{M}_t^{(\ell)}$, and the set of dangerous coordinates of $Y_t^{(\ell)}$ is not affected by the modification. 

Fix a column $j$ and consider $Y_t^{(\ell)}$ restricted to $\mathcal{C}_j$, the set of non-zero entries in column $j$. Analogous to the proof of \Cref{lem:weak_pairwise_bf_weak}, the increment of this vector is $(\alpha_\ell, \theta_\ell)$-pairwise independent with 
\[
\alpha_\ell := \frac{k_\ell}{\mu \eta_\ell} = \frac{6 k_\ell 2^\ell}{\mu} \quad \text{and} \quad \theta_\ell := \frac{\beta_\ell}{10} = \frac{b_\ell}{2^\ell \cdot 100 k_\ell} .
\]
Now applying \Cref{thm:general_freedman} to $Y_t^{(\ell)}$ restricted to $\mathcal{C}_j$ with parameters $\alpha = \alpha_\ell$, $\theta = \theta_\ell$, $B := b_\ell$ and $\lambda := 10/B$ (notice that this satisfies the requirement $\lambda \leq \theta/2$ as $b_\ell^2 \gg 2^\ell k_\ell$), we get that with high probability, the number of  level-$\ell$ rows in $\mathcal{C}_j$ that become dangerous is at most 
\begin{align*}
\frac{k_\ell}{e^{\lambda B}} + O \Big( \frac{\lambda \alpha \log n}{\theta} \Big) & = \frac{k_\ell}{e^{10}} + 
O \Big( \frac{\alpha_\ell \log n}{b_\ell\theta_\ell} \Big)  = 
\frac{k_\ell}{e^{10}} + O \Big(\frac{2^{2 \ell} k_\ell^2 \log n}{b_\ell^2 \mu} \Big) \\ & = \frac{k_\ell}{e^{10}} + O \Big(\frac{k_\ell \log^2 n}{C^4 k } \Big) \leq \frac{k_\ell}{100} ,
\end{align*}
where the last line uses that $b_\ell^2 = C^2 2^{2 \ell}k_\ell$ and  $\mu = b_0^2/\log n = C^2k/\log n $, and that $C$ is large enough and $k = \Omega(\log^2 n)$. 
\end{proof}

\smallskip
\noindent \textbf{Proof of \Cref{thm:bf_conj_large_k} for $k = \Omega(\log^2 n)$.} 
We first argue that with high probability, $\dim(W_t) \leq n_t/3$ and thus the algorithm never declares FAIL. 

We condition on the events in \Cref{lem:num_dang_bf_large_k}, which happen with high probability. 
As the sparsity for each column at  each level $\ell$ is at most $k_\ell$, there can be at most $n_t \cdot k_\ell / (2^\ell\cdot 10 k_\ell) = n_t/(10 \cdot   2^\ell )$ large level-$\ell$ rows. Hence there are at most $n_t/10$ large rows over all the levels $\ell$.

Similarly, the number of dangerous level-$L$ rows in any column $j$ is at most $k_L / 100$. 
As each dangerous level-$L$ row must have size at least $\mu$, the total number of dangerous $L$ rows is at most $n_t k_L/ (100 \mu) = n_t k/(100^L \cdot 100 \mu) \leq n_t/10$, by the choice of $L = \log_{100}(10 k/\mu)$.
Thus conditioned on the events in \Cref{lem:num_dang_bf_large_k}, the algorithm does not FAIL.

We now bound the discrepancy.
By design, for any row $i$, its discrepancy incurred while it is in level-$\ell$ is at most $O(b_\ell)$. As usual, once row $i$ becomes tiny (if it ever happens), its discrepancy increases further by at most $O(\sqrt{\mu \log n}) = O(\sqrt{k})$, with high probability. Thus the total discrepancy for each row $i$ is at most
\[
O(\sqrt{k}) + \sum_{\ell=0}^L O(b_\ell) = O(\sqrt{k}) + \sum_{\ell=0}^L O(2^L \sqrt{k_L}) = O(\sqrt{k}) .
\]
This gives the desired bound in \Cref{thm:bf_better} in the case of $k = \Omega(\log^2 n)$.  \qed

\subsection{Adapting to the Case of $k \ll \log^2 n$}
\label{subsec:small_k_bf_better}

The algorithm and the analysis for the case of $k \ll \log^2 n$ is essentially the same as the above, with only a few modifications in the setting of parameters that we describe below.

When $k \ll \log^2 n$, we set $b_0 :=  C k^{1/3} \log^{1/3} n$, and note that in this regime the tiny rows have size at most $\mu := \max(b_0, b_0^2/\log n) = b_0$. 

For each level $\ell$, we now set $b_\ell := C 2^\ell k_\ell^{1/3} \log^{1/3} n$, but keep the other parameters $k_\ell=k/(100^\ell),  \beta_\ell= b_\ell/(2^\ell \cdot 10 k_\ell)$ the same as before. We claim that \Cref{lem:num_dang_bf_large_k} still holds for these parameter settings. 

Indeed, when applying \Cref{thm:general_freedman} as in the proof of \Cref{lem:num_dang_bf_large_k}, we again have 
\[
\lambda_\ell := \frac{10}{b_\ell}, \quad \alpha_\ell := \frac{k_\ell}{\mu \eta_\ell} = \frac{6 k_\ell 2^\ell}{\mu} \quad \text{and} \quad \theta_\ell := \frac{\beta_\ell}{10} = \frac{b_\ell}{2^\ell \cdot 100 k_\ell}.
\]
The condition $\lambda_\ell \leq \theta_\ell/2$ still holds as,
\[
\frac{\lambda_\ell}{\theta_\ell} = O(1) \cdot \frac{2^\ell k_\ell}{b_\ell^2} = O(1) \cdot \frac{k_\ell^{1/3}}{2^\ell C^2 \log^{2/3}n} \ll 1 \tag{as $k_\ell \ll k \ll \log^2 n$}.
\]
Then, with high probability, the number of dangerous level-$\ell$ rows in column $j$ is bounded by
\begin{align*} \label{eq:dang_rows_bf_better}
\frac{k_\ell}{e^{\lambda_\ell B}} + O \Big( \frac{\lambda_\ell \alpha_\ell \log n}{\theta_\ell} \Big) & = \frac{k_\ell}{e^{10}} + 
O \Big( \frac{\alpha_\ell \log n}{b_\ell\theta_\ell} \Big)  = 
\frac{k_\ell}{e^{10}} + O \Big(\frac{2^{2 \ell} k_\ell^2 \log n}{b_\ell^2 \mu} \Big)  ,
\\
&= \frac{k_\ell}{e^{10}} + O \Big(\frac{2^{2 \ell} k_\ell^2 \log n}{C^3 2^{2 \ell} k_\ell^{2/3} k^{1/3} \log n} \Big) \leq \frac{k_\ell}{100},
\end{align*}
where the last line uses that $\mu=b_0$ and that $b_\ell=C 2^\ell k_\ell^{1/3} \log^{1/3} n$.

As before, conditioned on these events, the algorithm never declares FAIL, and the discrepancy of each row is at most
\[
O(b_0) + \sum_{\ell=0}^L O(b_\ell) = \sum_{\ell=0}^L O(k_\ell^{1/3} \log^{1/3} n) = O(k^{1/3} \log^{1/3} n).
\]
This gives the bound in \Cref{thm:bf_better} when $k \ll \log^2 n$.

%% file: Strong_Beck-Fiala.tex
\section{Improving the Beck-Fiala Bound Further}
\label{sec:bf_strong}

In this section, we further improve the Beck-Fiala bound and prove \Cref{thm:bf_strong}.
 
\BeckFialaStrong*

\noindent \textbf{High-Level Idea.} The algorithm and the analysis for \Cref{thm:bf_strong} are built upon the multi-layered extension in \Cref{sec:bf_better}, together with one more key idea that we explain below. 

For concreteness, let us consider the discrepancy bound $b_0$ at level $0$ --- which determines the overall discrepancy (as $b_\ell$ decreases geometrically with level $\ell$). 
The setting of $b_0$ is essentially determined by the  pairwise independence and negative drift parameters $\alpha_0$ and $\theta_0$ in the decoupling bound in \Cref{thm:general_freedman}. 
In particular, the additive term in \Cref{thm:general_freedman} is proportional to the ratio $\alpha_0/\theta_0$, 
and the greater this ratio gets, the worse our $b_0$ becomes. Previously, we set $\alpha_0 \approx k/\mu \gg 1$ and $\theta_0 \approx \beta_0 \approx b_0/k$, which results in the ratio $\alpha_0/\theta_0 \approx k^2/(b_0\mu)$. Our goal here will be to improve this ratio. So  
let us first see why these parameters are chosen in this way.

Observe that by the SDP constraints \eqref{sdp:pairwise-disc} and Theorem \ref{thm:sdp_feasibility}, the pairwise independence parameter $\alpha_0$ basically depends on the ratio of the number of medium rows $|\mathcal{M}_t|$ to $n_t$. As we can safely discard tiny rows of size less than $\mu$, we can safely set $\alpha_0 \approx k/\mu$. The setting of $\theta_0$ (which is $\approx \beta_0$) was chosen to ensure that $\beta_0 G_t(i) \ll b_0$ (in the regularized discrepancy). As each non-large row $i$ can have energy at most $k$, we can also safely set $\beta_0 \approx b_0/k$.

Both these bounds are the best possible for $\alpha_0$ and $\beta_0$ when considered individually, and cannot be improved in general. 
However, this setting of these parameters is pessimistic in that these worst case bounds cannot both hold simultaneously if the medium rows all have similar sizes.
More concretely, 
suppose all the medium rows have exactly the same size $s \in [\mu, k]$. Then there can be at most $\approx n_t k/s$ medium rows, which would give a better $\alpha_0 \approx k/s$. And we can also set $\beta_0\approx b_0/s$ more aggressively, as it still ensures that $\beta_0 G_t(i) \ll b_0$. 
With this setting, the ratio $\alpha_0/\theta_0$ would only be $\approx k/b_0$, improving upon the $k^2/b_0\mu$ ratio above by a $k/\mu \gg 1$ factor.

While the rows will have different sizes in general, we still can implement this idea by further dividing them into classes of similar sizes 
(within each level $\ell$), and set the parameters $\alpha$ and $\theta$ separately for each class (note that we can assume that $k\leq \log^2n$ and thus there are only $O(\log \log n)$ such size classes). One slight technicality is that as the size of each row changes over time, as its elements get colored, we will need to {\em dynamically} adjust parameters over time for each row. 
We then combine this idea with the idea of multiple levels from Section \ref{sec:bf_better} to obtain the overall algorithm.
This algorithm is still quite clean and the proof follows from the same template as before, except that the notation gets a bit more tedious, as we now need to track both the dangerous level and the size class for each row $i$.
We now give the details.

\subsection{A Multi-Layered Algorithm with Multiple Size Classes}
\label{subsec:alg_bf_strong}

We assume that $k < \log^2 n$, as the case of $k \geq \log^2 n$ is already handled optimally by \Cref{thm:bf_conj_large_k}. To simplify our presentation, we will not optimize the $\poly(\log \log n)$ factor in the theorem as they cannot be completely removed (for small $k$) with our current ideas.

Set the target discrepancy $b_0 := \widetilde{\Theta}(k^{1/2} + \log^{1/2} n)$, and  let $\mu := b_0$ as the size threshold for tiny rows. As usual, we can completely forget about a row once it is tiny, and we ignore these henceforth.

\smallskip
\noindent \textbf{Level of Non-Tiny Rows.} As in \Cref{sec:bf_better}, at each time $t$, we assign a level $\ell_t(i) \in \{0,1, \cdots, L\}$ to every row $i$ that is not tiny, where as before, $L:= \log_{100}(10 k/\mu) = O(\log \log n)$. The algorithm will guarantee that (1) the set of level-$\ell$ rows has column sparsity at most $k_{\ell} := k/(100)^\ell$, and (2) each row incurs discrepancy at most $\widetilde{O}(b_\ell)$, where the target discrepancy bound (at level $\ell$) is also $b_\ell:= b_0 =\widetilde{\Theta}(k_\ell^{1/2} + \log^{1/2} n)$ while it is at level $\ell$.

As in \Cref{sec:bf_better}, all non-tiny rows start at level $0$, and a level-$\ell$ row $i$ gets upgraded to level-$(\ell+1)$ when it becomes {\em dangerous} at level $\ell$. However, the notion of a dangerous row at a level $\ell$ is different from that in \Cref{sec:bf_better} and will be specified below. 
At any time $t$, a (non-tiny) row $i$ level-$\ell$ is called large if its size is more than $10 L k_\ell$.
Otherwise, we call it medium. The medium rows are further divided into $O(\log \log n)$ classes below, based on their current sizes, which we denote as $s_t(i)$. 

\smallskip
\noindent \textbf{Class of Medium Rows.} At time $t$, we say a level-$\ell$ medium row $i$ is in class $q \in [N]$, if its size 
\[s_t(i) \in \Big(\frac{10 L k_\ell}{2^q}, \frac{20 L k_\ell}{2^q} \Big],
\] 
i.e., within a factor of $\approx 2^q$ smaller than the largest allowable size of any level-$\ell$ medium row. 
For convenience, we denote the upper bound on the sizes of level-$\ell$ class-$q$ rows as $s^{(\ell)}_q := (20 L k_\ell)/2^q$.

Note that there are at most $N \leq \log (k/\mu) \leq \log\log n$ classes.  
We use $c_t(i) \in [N]$ to denote the class of a row $i$ at time $t$, and $\mathcal{M}_{t,q}^{(\ell)}$ denote the set of all level-$\ell$ medium rows that are in class $q$ at time $t$. 
Clearly, at any time $t$, a medium row $i$ can be exactly in one level $\ell$, and one class $q$. Note that over time, both its level $\ell$ and its class $q$ can only increase (as its size can only decrease). 

For a level-$\ell$ class-$q$ row $i$ at time $t$, we use $t^{(\ell)}_{i,q}$ to denote the time it reaches class $q$ at level $\ell$, and define its {\em regularized discrepancy} as
\begin{align} \label{eq:reg_disc_bf_strong}
Y_{t,q}^{(\ell)}(i) := \langle a_i, x_t - x_{t^{(\ell)}_{i,q}} \rangle + \beta^{(\ell)}_q G_t(i) ,
\end{align}
where $\beta^{(\ell)}_q := b_\ell/s^{(\ell)}_q$, and $G_t(i) := \sum_{j \in \mathcal{V}_t} a_i(j)^2 (1 - x_t(j)^2)$ is the energy as before.  

Note that for any level-$\ell$ class-$q$ row $i$, as $G_t(i) \leq s_t(i) \leq s^{(\ell)}_q$, the second term in \eqref{eq:reg_disc_bf_strong} can be at most $b_\ell$, and thus we have that $Y_{t,q}^{(\ell)}(i) \leq b_{\ell}$  at time $t = t^{(\ell)}_{i,q}$ when it reaches class $q$ at level $\ell$.
We let $E^{(\ell)}_{t,q} \in \mathbb{R}^{|\mathcal{M}_{t,q}^{(\ell)}| \times n}$ denote the matrix whose $i$th row is $a_i - 2 \beta^{(\ell)}_q a_i^2 x_t$. 

\smallskip
\noindent 
{\bf Dangerous Rows.}
We say that a level-$\ell$ class-$q$ row $i$ is dangerous at time $t$ if $Y_{t,q}^{(\ell)}(i) \geq 2 b_\ell$. When a row turns dangerous in level $\ell<L$, it is upgraded to level $(\ell+1)$ (regardless of its class $q$), and at this point its regularized discrepancy resets with its new parameters of level and class. 
As before, dangerous level-$L$ rows do not get upgraded as they will be blocked. 

A level-$\ell$ row $i$ might also get upgraded to a higher size class from its current size class $q$ (as its size decreases), and and this point its regularized discrepancy in \eqref{eq:reg_disc_bf_strong} resets again.

We now give the algorithm.

\smallskip
\noindent \textbf{The Algorithm.} We choose the subspace $W_t$ and the matrices $E_s$ in the SDP at time $t$ as follows. 

\begin{enumerate}
    \item Let $W_t$ denote the subspace spanned by (i) all the large rows $a_i$ at all levels, (ii) the $i$th row of $E_L(t)$ for all dangerous level-$L$ rows $i$, (iii) the vector $x_t$. If $\dim(W_t) > n_t/3$, declare FAIL. 

    \item Solve the SDP \eqref{sdp:jj_new}-\eqref{sdp:psd_new} with subspace $W_t$, the matrix $E_{\ell,q} = E^{(\ell)}_{t,q}$ and parameters  $\delta=1/3, \kappa = 1/6$, $\eta = 1/6$, and $\eta_{\ell,q} = \eta/(N L)$ for all levels $\ell \in \{0, \cdots, L\}$ and classes $q \in [N]$.
    Use the resulting SDP solution $U_t$ to find $v_t$ as in \eqref{eq:find_vt}, and update the coloring $x_t$ by $v_t \sqrt{dt}$.
\end{enumerate}

\subsection{The Analysis}

As before, we bound the number of rows that get dangerous at each level within each column.

\begin{lemma}[Number of dangerous rows]
\label{lem:num_dang_bf_strong}
With high probability, for all columns $j \in [n]$, levels $\ell \in \{0,\cdots,L\}$, and classes $q \in [N]$, at most $k_\ell/(100 N)$ level-$\ell$ class-$q$ rows in column $j$ ever become dangerous during the algorithm. In particular, at most $k_\ell$ rows in column $j$ ever reaches level $\ell$.
\end{lemma}

\begin{proof}
The proof is analogous to that of \Cref{lem:num_dang_bf_large_k}, by inductively showing that with high probability, at most $k_\ell/(100N)$ level-$\ell$ rows in column $j$ ever become dangerous throughout the algorithm. 
Let us assume inductively that at most $k_\ell$ rows in column $j$ ever reaches some level $\ell$. 

As before, we would like to apply the decoupling bound in \Cref{thm:general_freedman} to the random process $Y_{t,q}^{(\ell)}(i)$ for all level-$\ell$ class-$q$ rows $i$. 
As we are inductively assuming that the level-$\ell$ rows are $k_{\ell}$-column sparse, and as each class-$q$ row has size $>s^{(\ell)}_q/2$, there can be at most $n_t k_{\ell} / (s^{(\ell)}_q/2)$ level-$\ell$ class-$q$ rows at any time. 
Therefore, analogous to the proof of \Cref{lem:weak_pairwise_bf_weak}, the increment of $Y_{t,q}^{(\ell)}$ is $(\alpha^{(\ell)}_q, \theta^{(\ell)}_q)$-pairwise independent with
\[
\alpha^{(\ell)}_q := \frac{k_\ell}{(s^{(\ell)}_q/2) \eta_{\ell,q}} = \frac{12 N L k_\ell}{s^{(\ell)}_q } \quad \text{ and } \quad \theta^{(\ell)}_q := \frac{\beta^{(\ell)}_q}{10} = \frac{b_\ell}{10 s^{(\ell)}_q} .
\]

We now apply \Cref{thm:general_freedman} to $Y^{(\ell)}_{t,q}$ restricted to $\mathcal{C}_j$ with parameters $\alpha = \alpha^{(\ell)}_q$, $\theta = \theta^{(\ell)}_q$, $B := b_\ell$ and $\lambda := (10 \log\log\log n)/B$, which satisfies the requirement $\lambda \leq \theta/2$ as 
\[
\frac{\lambda}{\theta} = O(1) \cdot \frac{(\log\log\log n) s_q^{(\ell)}}{b_\ell^2} = O(1) \cdot \frac{(\log\log\log n) L k_{\ell}}{b_\ell^2 2^q} \ll 1 ,
\]
where the last inequality above follows as $b_\ell = \widetilde{\Theta}(k_{\ell}^{1/2} + \log^{1/2} n)$, and thus $b_\ell^2 \gg L k_\ell \log \log \log n$. 
This gives that with high probability, the number of dangerous level-$\ell$ class-$q$ rows is at most 
\begin{align*}
\frac{k_\ell}{e^{\lambda B}} + O \Big( \frac{\lambda \alpha \log n}{\theta} \Big) 
& = \frac{k_\ell}{(\log \log n)^{10}} + O \Big(\frac{\log\log\log n}{b_\ell} \cdot \frac{N L k_\ell}{s^{(\ell)}_q} \cdot \frac{s^{(\ell)}_q}{b_\ell} \cdot \log n\Big) \\
&= \frac{k_\ell}{(\log \log n)^{10}} + O \Big(\frac{(\log\log\log n) N L k_\ell \log n}{b_\ell^2} \Big) \leq \frac{k_\ell}{100 N} ,
\end{align*}
where the last inequality follows as $b_\ell^2 = \widetilde{\Omega}(\log n)$ which is also $ \widetilde{\Omega}(NL \log n)$, since both $N$ and $L$ are at most $O(\log \log n)$.
\end{proof}

\noindent \textbf{Wrapping Things Up.} Now we can wrap things up and prove \Cref{thm:bf_strong}.

\begin{proof}[Proof of \Cref{thm:bf_strong}]
By \Cref{lem:num_dang_bf_strong}, within each column $j$, the number of dangerous level-$\ell$ rows across all size classes is at most $k_\ell/100$ with high probability. As in \Cref{sec:bf_better}, this implies that the algorithm never declares FAIL with high probability. 

For any row $i$, by design of the algorithm, it can incur discrepancy at most $b_\ell$ at each class $q$ while it is at level $\ell$, and at most $\mu = b_0$ after it becomes tiny. Thus the total discrepancy for each row is at most 
\[
b_0 + \sum_{\ell=0}^L O(N b_\ell) \leq \widetilde{O}(k^{1/2} + \log^{1/2} n) ,
\]
where the inequality above uses our setting of $b_\ell = \widetilde{\Theta}(L(k_\ell^{1/2} + \log^{1/2} n))$, and that both $L$ and $N$ are $O(\log\log n)$ and are subsumed in $\widetilde{O}(\cdot)$.
\end{proof}

%% file: conclusions.tex
\section{Concluding Remarks}
\label{sec:conclusion}
Our $\widetilde{O}(k^{1/2}+ \log^{1/2}n)$ bound for the Beck-Fiala problem gives evidence that the Beck-Fiala Conjecture may be true for any $k$ (possibly up to a $\poly(\log \log n)$ factor).
Currently, the additive $\widetilde{O}(\log^{1/2}n)$ loss in \Cref{thm:bf_strong}  prevents us from improving over the simple $O(k)$ bound for $k\leq \log^{1/2}n$. Also recall that for Beck-Fiala instances with $k=\log^{1/2}n$, the $O(k)$ bound translates to an $O(k^{1/2})=O(\log^{1/4} n)$ bound for the corresponding Koml\'os instance, and thus our bound for the Koml\'os problem currently gets stuck at $\widetilde{O}(\log^{1/4} n)$.

It would be very interesting if this additive $\widetilde{O}(\log^{1/2}n)$ loss could be improved or even removed. We briefly discuss the reason for this loss in our current analysis, and also mention some related questions below.

\smallskip
\noindent
{\bf The Reason for the Loss.}
Currently, this additive loss arises due to the concentration bound required to ensure that the number of dangerous rows is not too large for {\em every} column and {\em every} time $t$. We used this to bound the number of dangerous rows for any time $t$, irrespective of which subset of $n_t$ columns were alive at that time. 

This worst case control for {\em every} column is needed because currently we have no way to control how the elements get colored as the algorithm progresses. In particular, we need to assume that the alive elements at each time $t$ may be chosen completely adversarially. But unfortunately,  the $\log^{1/2} n$ additive loss is provably unavoidable 
under this adversarial assumption.

It seems plausible that even a mild control on which subset of elements are alive may suffice to give $\widetilde{O}(k^{1/2})$ and $\widetilde{O}(1)$ bounds for Beck-Fiala and Koml\'os problems.
In particular, notice that if our target discrepancy is $ck^{1/2}$, then for a typical column the expected number of dangerous rows is at most $\approx k \exp(-c^2)$ which is $\ll 1$ already for 
$c= \Omega(\log \log n)$ (as we can assume that 
$k \leq  \log^2n$).

We now state some other related open questions.

\smallskip
\noindent{\bf
Koml\'os Problem with Bounded Rows and Columns.} 
Consider the special case of the Koml\'os problem where both the rows and columns have unit length. 
An interesting example of such a setting is the case where the columns form some orthonormal basis. A bound on the  discrepancy for such instances would imply the same bound on the covering radius of an arbitrarily rotated unit cube with respect to the integer lattice. Currently, nothing better than the general bound is known for this special case.

Interestingly, for the Beck-Fiala version of this problem, i.e., where the rows and columns both have $k$ entries with magnitude $k^{-1/2}$, an $O((\log \log n)^{1/2})$ bound follows from the results here (or in \cite{BJ25b}). This is because for $k\leq \log^2 n$ the Lov\'asz Local Lemma (LLL) implies an $O(\sqrt{\log k}) = O((\log \log n)^{1/2})$ bound, and for $k\geq \log^2 n$,  \Cref{thm:bf_conj_large_k} gives an $O(1)$ bound.

Even though a Koml\'os instance can be viewed as a combination of Beck-Fiala instances at various scales, it is unclear how to get an analogous $O(\poly(\log \log n))$ bound. The problem is that while the LLL works well for small $k$, say $k = \log^{1/2}n$, it loses an $O(\sqrt{\log n})$ factor when $k$ is large (say $n^{0.1}$). On the other hand, the methods in this paper work well for large $k$, but lose an $\widetilde{O}(\log^{1/4} n)$ factor for $k= \log^{1/2}n$. 
Making progress on this problem will probably require a unified algorithm that combines the ideas for LLL together with the techniques here. Such an algorithm will be extremely interesting and is likely to have other useful applications.

\medskip
\noindent {\bf Prefix Discrepancy Problems.}
It would be very interesting if the  techniques here can also be useful for prefix discrepancy problems.

In the prefix Koml\'os problem, given a matrix $A \in \mathbb{R}^{m \times n}$ with unit-length columns, we want to find a coloring $x \in \{-1,1\}^n$ to minimize the discrepancy over all the prefixes  (first $j$ columns of $A$ for every $j\leq n$).
Extending his approach from \cite{Ban98}, Banaszczyk showed a non-constructive bound of $O(\sqrt{\log n})$ \cite{Ban12}. Interestingly, the best algorithmic bound is still only $O(\log n)$. Can we improve Banaszczyk's non-constructive bound, or give a better algorithmic bound than $O(\log n)$?
The same question applies for the Beck-Fiala case.  

Any improvements to these prefix problems might also give better bounds for Tusn\'ady's problem, where one wants to minimize the combinatorial discrepancy of all axis-parallel rectangles for a set of $n$ points in $[0,1]^d$, using ideas from \cite{Nik17,BG17}.  
Another related question is the vector balancing problem, closely related to the Steinitz problem. Here we are given vectors $v_1,\ldots,v_n \in \R^m$, each with $\|v_i\|_\infty \leq 1$, and our goal is to find a coloring to minimize the discrepancy of each prefix. The best bound is $O(\sqrt{m \log n})$  due to Banaszczyk, and the corresponding algorithmic version was obtained in \cite{BG17}. Can these bounds be improved?

%% file: appendix.tex
\section{Missing Proofs and Technical Details}

\subsection{Banaszczyk's $O(\log^{1/2} n)$ Bound for the Koml\'os Problem}
\label{subsec:bana_bound_komlos}
In this section, we give details for the algorithmic proof of Banaszczyk's bound from \Cref{subsec:bana_bf_overview}. 
In fact, we recover Banaszczyk's more general $O(\log^{1/2} n)$ bound for the Koml\'os problem. 
These ideas have appeared either directly or indirectly in \cite{BG17,BDG19,BLV22}, but we give a simple and self-contained proof here for completeness. 

\subsubsection{An Algorithm for Banaszczyk's Koml\'os Bound}
\label{subsubsec:alg_bana_komlos}
Let $A$ be an $m\times n$ input matrix with each columns $a^j$ satisfying $\|a^j\|_2\leq 1$. 
Consider the following algorithm using the framework and notation from \Cref{subsec:basic_framework,subsec:sdp_bana}. 

At each time $t$, repeat the following until $n_t \leq  10$.
\begin{enumerate}
\item Call a row $i$ large (at time $t$)  if $\sum_{j \in \mathcal{V}_t} a_i(j)^2 \geq 4$. Let $W_t$ denote the subspace spanned by all the large rows, and the additional (single) vector $x_t$. 
    \item Solve the SDP for Banaszczyk in \eqref{sdp:jj}-\eqref{sdp:psd} with $W = W_t$, $\kappa = 1/4$ and $\eta=1/4$. Use the resulting solution $U_t$ to find $v_t$ as in \eqref{eq:find_vt}, and update the coloring $x_t$ by $v_t \sqrt{dt}$.
\end{enumerate}

\subsubsection{Analysis} 
\label{subsubsec:analysis_Banaszczyk_komlos}
We first show that the SDP is always feasible and then bound the discrepancy.
\begin{claim}
    The SDP is feasible at each time $t$.
\end{claim}
\begin{proof}
Fix some time $t$. As $\sum_i a_i(j)^2\leq 1$ for each column $j$, we have that $\sum_{j \in \mathcal{V}_t}  \sum_{i \in [n]} a_i(j)^2 \leq n_t$. Thus by an averaging argument, the number of large rows is at most $n_t/4$. 

So $\dim(W_t) \leq n_t/4 + 1 \leq n_t/2$, as $n_t\geq 10$, and thus $\delta \leq 1/2$. As we set $\kappa=\eta=1/4$ (and as $q=0$) \Cref{thm:sdp_feasibility} implies that the SDP is feasible.
 \end{proof}
Notice that as $x_t$ lies in the subspace $W_t$, we have $v_t \perp x_t$, and thus $\|x_t\|^2=t$ as discussed earlier and the algorithm eventually terminates and is well-defined.

\smallskip
\noindent \textbf{Discrepancy Bound.}
We now show that with high probability all rows have discrepancy $O(\log^{1/2} n)$. 
To this end, we use the following Freedman-type  concentration inequality for super-martingales. 

\begin{fact}[Lemma 2.2 in \cite{Ban24}]  \label{fact:freedman_conc_komlos}
Let $\{Z_t: t = 0,1, \cdots\}$ be a sequence of random variables with increments $\Delta Z_t := Z_t - Z_{t-1}$, such that $Z_0$ is deterministic and $\Delta Z_t \leq M$ for all $t \geq 1$.\footnote{Strictly speaking, this appears in \cite{Ban24} with $M=1$, but the version here follows directly by scaling $\Delta Z_t$.
For our purposes, $M$ will be $O(\sqrt{dt}) = 1/\poly(n)$ and hence arbitrarily small.} If for all $t \geq 1$, 
\[
\E_{t-1}[\Delta Z_t] \leq - \delta\, \E_{t-1}[(\Delta Z_t)^2]
\]
holds with $0 < \delta < 1/M$, where $\E_{t-1}[\cdot]$ denotes $\E[\cdot | Z_1, \cdots, Z_{t-1}]$. Then for all $\xi \geq 0$, 
\[
\p\big( Z_t - Z_0 > \xi \big) \leq \exp(- \delta \xi). \]
\end{fact}
Fix some row $i$. Notice that as long as the row is large, $v_t$ is orthogonal to it, and its discrepancy stays $0$. Thus, it suffices to bound the discrepancy from the time the row becomes small. 
This follows by \Cref{lem:algo_bana_komlos} below, which we
state in a slightly more general form for use later on.

\begin{restatable}{lemma}{AlgoBanaszczyk}
\label{lem:algo_bana_komlos}
Let $y_t \in [-1,1]^n$ be a random process for $t \in [0,T]$ with increment $d y_t = v_t \sqrt{d t}$, where $v_t \in \R^n$ is an arbitrary mean-zero random vector\footnote{The distribution of $v_t$ is allowed to depend on the outcome of all $v_{t'}$ with $t' < t$.} such that $\E[v_t v_t^\top] \preceq O(1) \cdot  \diag(\E[v_t v_t^\top])$. Then for any vector $a \in \mathbb{R}^n$, with probability at least $1 - 1/\poly(n)$,
\[
\big|\langle a, y_T - y_0 \rangle \big| 
\leq O(\|a\|_2 \log^{1/2} n) .
\]
\end{restatable}

\begin{proof}[Proof of \Cref{lem:algo_bana_komlos}]
By rescaling $a$, we may assume without loss of generality that $\|a\|_2 = 1$. 
We write $a = a_b + a_s$, where $a_b$ contains the big entries of $a$ with magnitude at least $1/\sqrt{\log n}$, i.e., $a_b(j) = a(j) \cdot \mathbf{1}(|a(j)| \geq 1/\log^{1/2} n)$, and $a_s$ contains the remaining small entries. 

\smallskip
\noindent \textbf{Discrepancy of Big Entries.} As $\|a\|_2 = 1$, clearly $\|a_b\|_2 \leq 1$, so $a_b$ has at most $\log n$ non-zero entries.  Thus, by the Cauchy-Schwartz inequality,
\[
\big|\langle a_b, y_T - y_0 \rangle \big| \leq \|a_b\|_2 \cdot \Big(\sum_{j: a_b(j) \neq 0} (y_T(j) - y_0(j))^2\Big)^{1/2} = O(\log^{1/2} n).
\]

\noindent \textbf{Discrepancy of Small Entries.} 
We now bound the discrepancy due to the small entries $a_s$. We will do this by applying \Cref{fact:freedman_conc_komlos} to a suitable process $Z_t$ that we define next.

At time $t$, let $G_t := \sum_{j=1}^n a_s(j)^2 (1 - y_t(j)^2)$ denote the {\em energy} of $a_s$. Set $\beta := \log^{1/2} n$. We define the {\em regularized} discrepancy 
\[
Z_t := \langle a_s, y_t \rangle + \beta G_t . 
\]
Notice that $Z_t \geq \langle a_s, y_t \rangle$, as $G_t\geq 0$ for all $t$. Moreover,  $G_t \leq 1$ as $\|a_s\|_2 \leq 1$ and $y_t \in [-1,1]^n$. So intuitively, $Z_t$ should be viewed as a good proxy and an upper bound for discrepancy.

Using $dy_t = v_t \sqrt{dt}$, a direct computation gives that
\begin{align} \label{eq:dZ_t_komlos}
d Z_t = \langle a_s - 2 \beta y_t a_s^2, v_t \rangle \sqrt{dt} - \beta \langle a_s^2, v_t^2 \rangle d t,
\end{align} 
where recall that we use $a_s^2$ to denote the vector with entries $a_s(j)^2$, and similarly for $v_t^2$.
Note that 
\[\E[d Z_t] = - \beta \,\, \E \langle a_s^2, v_t^2 \rangle d t.\]
Moreover, the second moment of $dZ_t$ can be upper bounded as 
\begin{align*}
\E[(d Z_t)^2] 
& = \E  \langle a_s - 2 \beta y_t a_s^2, v_t \rangle^2 d t + O((dt)^{3/2}) \\
& \leq  O(1) \cdot \E  \langle (a_s - 2 \beta y_t a_s^2)^2, v_t^2 \rangle \leq  O(1) \cdot \E \langle a_s^2, v_t^2 \rangle d t,
\end{align*}
where the first inequality uses that $\E[v_t v_t^\top] \preceq O(1) \cdot \diag(\E[v_t v_t^\top])$, the second inequality follows as $|a_s(j)| \leq 1/\sqrt{\log n}$ and $ |y_t(j)| \leq 1$  for each $j$, and thus
$|2 \beta y_t(j) a_s(j)^2| = O(|a_s(j)|)$.

So $\E[d Z_t] \leq  -\Omega(\beta)\E[(d Z_t)^2]$, and applying \Cref{fact:freedman_conc_komlos} with $\delta = \Omega(\beta)$ and $\xi = \Omega(\beta^{-1} \log n) $ gives 
\[
\Pr[Z_T - Z_0 \geq \xi] = \exp(-\delta \xi ) = 1/\poly(n). 
\]
As $\xi = O(\log^{1/2} n)$, this implies that $Z_T \leq Z_0 + O(\log^{1/2} n)$  with high probability.
As $\langle a_s, y_T \rangle \leq Z_T$ and $Z_0 = \langle a_s, y_0 \rangle + \beta G_0 \leq \langle a_s, y_0 \rangle + \log^{1/2} n$, this implies that $\langle a_s, y_T - y_0 \rangle \leq O(\log^{1/2} n)$.
\end{proof}

\Cref{lem:algo_bana_komlos} implies that with high probability, the discrepancy of each row, from the time it becomes small, is at most $O(\log^{1/2} n)$. This gives Banaszczyk's bound for the Koml\'os problem.

\subsection{Proof of SDP Feasibility}
\label{subsec:sdp_feasibility}

We prove \Cref{thm:sdp_feasibility} in this subsection, where we restate it with the SDP constraints included. 
The proof is a slight generalization of the argument in \cite[Section 4]{BG17}. 

\begin{restatable}
{theorem}{SubIsoVecMultiClass}\label{thm:sub-isotropic-SDP_Multi-Class}
Let $W \subset \R^h$ be a subspace with dimension $\dim(W) = \delta h$, and $E_s \in \mathbb{R}^{r_s h \times h}$ for all $s \in [q]$ be a set of matrices with $r_s \geq 1$. 
Then for any $0 < \kappa, \eta, \eta_s < 1$, where $s \in [q]$, such that $\eta + \kappa + \sum_{s=1}^q \eta_s \leq 1-\delta$, there is an $h \times h$ PSD matrix $U$ satisfying the following:

(i) $\langle ww^\top, U \rangle = 0$ for all $w \in W$, 

(ii) $U_{ii} \leq 1$ for all $i \in [h]$, 

(iii) $\Tr(U) \geq \kappa h$, 

(iv) $U \preceq \frac{1}{\eta} \diag(U)$, and 

(v) $E_sU E_s^\top \preceq \frac{r_s}{\eta_s} \diag(E_s U E_s^\top)$ for all $s \in [q]$.

Furthermore, such a PSD matrix $U$ can be computed by solving a semidefinite program (SDP). 
\end{restatable}

\begin{proof}
Consider the following SDP
\begin{equation}
\label{eq:sdp_primal} \tag{Primal SDP}
\begin{aligned}
    \max \quad & \Tr(X) \\
    s.t. \quad & X \cdot ww^\top = 0 \ , && \text{for all } w \in W ,\\
    & X \cdot e_i e_i^\top \leq 1 \ , &&\text{for all } i \in [h] , \\
    & X \preceq \frac{1}{\eta} \diag(X) , \\
    & E_s X E_s^\top \preceq \frac{r_s}{\eta_s} \diag(E_s X E_s^\top ) \ , \qquad && \text{for all } s \in [q] , \\
    & X \succeq 0 .
\end{aligned}
\end{equation}
It suffices to show that the optimal value of \eqref{eq:sdp_primal} is at least $\kappa h$, and we do so by considering the dual SDP. 
Let $\gamma_w$, $\alpha_i \geq 0$, $H \succeq 0$, and $G_s \succeq 0$ be the Lagrangian multipliers for the constraints of \eqref{eq:sdp_primal}. 
Note that $H \in \R^{h\times h}$ and $G_s \in \R^{r_sh \times r_sh}$.

Then the dual SDP is given by 
\begin{equation}
\label{eq:sdp_dual} \tag{Dual SDP}
\begin{aligned}
& \min \quad \sum_{i \in [h]} \alpha_i \\
& \textrm{s.t.} \sum_{w \in W} \gamma_w ww^\top + \sum_{i \in [h]} \alpha_i e_i e_i^\top + H - \frac{1}{\eta} \diag(H) + \sum_{s \in [q]} (E_s^\top G_s E_s - \frac{r_s}{\eta_s} E_s^\top \diag(G_s) E_s) \succeq I , \\
& \quad H \succeq 0 , \  G_s \succeq 0 \text{ for all $s \in [q]$} , \ \text{and } \alpha_i \geq 0 \text{ for all } i \in [h] .
\end{aligned}
\end{equation}

Note that strong duality holds here, as \eqref{eq:sdp_dual} is strictly feasible (e.g., \cite[Section 5.9]{BV04book}). For instance, the dual feasible solution $H = G_s = I$ for all $s \in [q]$, $\gamma_w = 0$ for all $w \in W$, and $\alpha_i = \sum_{s \in [q]} (2r_s/\eta_s) \Tr(E_s^\top E_s) + 2/\eta$ for all $i \in [h]$ is in the interior of the dual feasible region. 

Our goal is to prove that for an arbitrary feasible solution to \eqref{eq:sdp_dual}, its objective value is at least $\kappa h$.
By strong duality, this would imply that there must be a solution to \eqref{eq:sdp_primal} with objective value at least $\kappa h$.

To do so, we recall the following fact from \cite{BG17}.
\begin{fact}[\cite{BG17}] \label{fact:subspace_diag}
There exists a subspace $W_H \subset \R^h$ with dimension $\dim(W_H) \geq (1 - \eta) h$ such that for all vectors $v \in W_H$, one has  $v^\top H v \leq \frac{1}{\eta} v^\top \diag(H) v$ .
\end{fact}
We slightly generalize this fact in the following claim, whose proof is deferred to later.
\begin{claim} \label{claim:subspace_diag_general}
For any $s \in [q]$, there exists a subspace $W_s \subset \R^h$ with dimension $\dim(W_s) \geq (1 - \eta_s) h$ such that for all vectors $v \in W_s$, one has  $v^\top E_s^\top G_s E_s v \leq \frac{r_s}{\eta_s} v^\top E_s^\top \diag(G_s) E_s v$ . 
\end{claim}
Using \Cref{fact:subspace_diag} and \Cref{claim:subspace_diag_general}, we find subspaces $W_H$ of dimension at least $(1 - \eta)$ and $W_s$ for all $s \in [q]$ of dimensions at least $(1 - \eta_s)h$ each with the corresponding guarantees.
This means for any vector  $v \in W_{\mathsf{neg}} := W^\perp \cap W_H \cap \big(\bigcap_{s =1}^q W_s \big)$, we have
\begin{align} \label{eq:neg_subspace_SDP}
v^\top \Big( \sum_{w \in W} \gamma_w ww^\top + H - \frac{1}{\eta} \diag(H) + \sum_{s \in [q]} \big(E_s^\top G_s E_s - \frac{r_s}{\eta_s} E_s^\top \diag(G_s) E_s \big)  \Big) v \leq 0 .
\end{align}
Now consider an arbitrary feasible solution $H \succeq 0, G_s \succeq 0, \alpha_i \geq 0$ to \eqref{eq:sdp_dual}.
Note that $\dim(W_{\mathsf{neg}}) \geq \big (1 - \delta - \eta - \sum_{s \in [q]} \eta_s \big) h \geq \kappa h$, so there exists a set of orthonormal vectors $v_1, \cdots, v_{\kappa h} \in W_{\mathsf{neg}}$. Then we can bound 
\begin{align*}
\sum_{i \in [h]} \alpha_i  \geq \sum_{j \in [\kappa h]} \sum_{i \in [h]} \alpha_i e_i e_i^\top \bullet v_j v_j^\top \geq \sum_{j \in [\kappa h]} v_j v_j^\top \bullet I = \kappa h ,
\end{align*}
where the first inequality uses that $\sum_{j \in [\kappa_h]} v_j v_j^\top \preceq I$, the second inequality follows from the constraint of \eqref{eq:sdp_dual} and the inequality \eqref{eq:neg_subspace_SDP}. 
This shows that \eqref{eq:sdp_dual} has value at least $\kappa h$, which proves the theorem. 
\end{proof}

We are now left to prove  \Cref{claim:subspace_diag_general}. 

\begin{proof}[Proof of \Cref{claim:subspace_diag_general}]
We may assume without loss of generality that all diagonal entries of $G_s$ are strictly positive.\footnote{Otherwise, we first prove the statement for $G_s + \varepsilon I$ for all positive $\varepsilon$ and then take the limit $\varepsilon \rightarrow 0$.} Define $\widetilde{G}_s := \diag(G_s)^{- 1/2} G_s \diag(G_s)^{- 1/2}$. 
Then since $G_s$ has dimension $r_s h$, we have
\[
\Tr(\widetilde{G}_s) \leq r_s h .
\]
Let  $W_{s,G} \subset \R^{r_s h}$ be the span of all eigenvectors of $\widetilde{G}_s$ with eigenvalues at most $r_s/\eta_s$. 
Then the above bound implies that $\dim(W_{s,G}) \geq r_s h - \eta_s h$, as there can be  at most $\eta_s h$ eigenvalues exceeding $r_s/\eta_s$. 
Define $W'_{s,G} := \diag(G_s)^{-1/2} W_{s,G}$, which also has $\dim(W'_{s,G}) \geq r_s h - \eta_s h$. 
Then for any vector $u \in W'_{s,G}$, we have
\begin{align*}
u^\top G_s u & = (\diag(G_s)^{1/2} u)^\top \widetilde{G}_s (\diag(G_s)^{1/2} u )\\
& \leq \frac{r_s}{\eta_s} \cdot (\diag(G_s)^{1/2} u)^\top \diag(G_s)^{1/2} u = \frac{r_s}{\eta_s} u^\top \diag(G_s) u ,
\end{align*}
where the inequality above follows as $\diag(G_s)^{1/2} u \in W_{s,G}$ and all eigenvalues of $\widetilde{G}_s$ in the subspace  $W_{s,G}$ are at most $r_s/\eta_s$. 

We view $E_s: \mathbb{R}^h \rightarrow \mathbb{R}^{r_s h}$
as a linear map with kernel $\mathsf{Ker}(E_s)$. 
Note that any vector $v \in \mathsf{Ker}(E_s)$ clearly satisfies the statement of the claim, since both sides of the inequality are $0$. 
Any vector $v \in \mathbb{R}^h$ such that $E_s v \in W'_{s,G}$ also satisfies the claim statement. 
Therefore, the subspace 
\[
W_s := \mathsf{Ker}(E_s) \oplus E_s^{-1}(E_s(\R^h) \cap W_{s,G})
\]
satisfies $\dim(W_s) \geq (1 - \eta_s) h$ and the statement of the claim. 
\end{proof}

\subsubsection{The $O(r_s)$ Factor is Necessary for Affine Spectral-Independence}
\label{subsubsec:necessity_r_s}

Here we show that the extra factor of $O(r_s)$ on the right-hand side of constraints \eqref{sdp:pairwise-disc} is needed, as otherwise the SDP might not be feasible. 

Consider the following example. Let $E \in \mathbb{R}^{rh \times h}$ be the vertical stacking of $r$ identity matrices, i.e., $E = (I, I, \cdots, I)^\top$. We claim that the constraint 
\begin{align} \label{sdp:pairwise_counter_example}
E X E^\top \preceq \frac{1}{\eta} \diag(E X E^\top)  
\end{align}
has no PSD solution $X \succeq 0$ with any non-zero diagonal entry whenever $r > 1/\eta$. 
To see this, note that $EXE$ is the $r \times r$ block matrix with every $h \times h$ block being $X$, i.e.,
\[
EXE^\top = \left( 
\begin{matrix}
    X & X &\cdots & X \\
    \vdots & \vdots & \vdots & \vdots \\
     X & X &\cdots & X 
\end{matrix}
\right) .
\]
Suppose $X$ has a non-zero diagonal entry $X_{i,i} > 0$, then consider the test vector $u = (e_i^\top, e_i^\top, \cdots, e_i^\top)^\top$ which is the vertical stacking of $r$ of the $i$th standard orthonormal basis vector $e_i$, which satisfies
\begin{align*}
u^\top EXE^\top u = r^2 X_{i,i} \quad \text{and} \quad u^\top \Big(\frac{1}{\eta} \diag(E X E^\top)  \Big) u = \frac{r}{\eta} X_{i,i} .
\end{align*}
Then clearly constraint \eqref{sdp:pairwise_counter_example} is violated whenever $r > 1/\eta$. 

This example shows that an extra factor depending on the ratio of the row and column dimensions of $E_s$ is needed in constraints \eqref{sdp:pairwise-disc}. In particular, if $\eta_s$ is set to be a fixed constant, then a factor of $O(r_s)$ is indeed necessary for the feasibility of our new SDP.